%% file: main.tex
\documentclass[english,11pt, a4paper]{article}

\usepackage[T1]{fontenc}
\usepackage[latin9]{inputenc}
\usepackage{verbatim}
\usepackage{amscd}
\usepackage{mathtools}
\usepackage{lmodern}
\usepackage{amsthm,amsmath,amsfonts,amssymb,esint}
\usepackage{graphicx}
\usepackage{enumerate,enumitem}
\usepackage{authblk}
\usepackage{setspace}

\let\OLDthebibliography\thebibliography
\renewcommand\thebibliography[1]{
  \OLDthebibliography{#1}
  \setlength{\parskip}{2pt}
  \setlength{\itemsep}{0pt plus 0.3ex}
}

\usepackage[activate={true,nocompatibility},final,tracking=true,kerning=true,spacing=true,factor=1100,stretch=10,shrink=10]{microtype}
\microtypecontext{spacing=nonfrench}

\makeatletter

\numberwithin{equation}{section}
\numberwithin{figure}{section}
\theoremstyle{plain}
\newtheorem{thm}{\protect\theoremname}
  \theoremstyle{plain}
  \newtheorem{lemma}[thm]{\protect\lemmaname}
    \newtheorem{prop}[thm]{\protect\propname}

\numberwithin{thm}{section}

\newtheorem{cor}[thm]{Corollary}

\theoremstyle{remark}
\newtheorem*{rem}{Remark}

\makeatother

\usepackage{babel}
\providecommand{\propname}{Proposition}

\providecommand{\lemmaname}{Lemma}
\providecommand{\theoremname}{Theorem}

\renewcommand{\Im}{\imag}
\renewcommand{\Re}{\real}

\newcommand{\ee}{\epsilon}

\newcommand{\HH}{\mathbb{H}}
\newcommand{\DD}{\mathbb{D}}
\newcommand{\CC}{\mathbb{C}}

\newcommand{\vp}{\varphi}

\newcommand{\diam}{\operatorname{diam}}

\newcommand{\dist}{\operatorname{dist}}
\newcommand{\SLE}{\operatorname{SLE}}

\let \le \leqslant

\let \ge \geqslant

\let \epsilon \varepsilon
\let \phi \varphi
\let \vp \varphi

\def\field{function}

\input{yilin}

\usepackage[dvipsnames]{xcolor}

\title{Interplay between Loewner and Dirichlet energies \\ via conformal welding and flow-lines}
\author{Fredrik Viklund\thanks{KTH Royal Institute of Technology, Stockholm, Sweden. \protect\url{fredrik.viklund@math.kth.se}} \, and Yilin Wang\thanks{MIT, Cambridge, MA, USA.  \protect\url{yilwang@mit.edu}}}
\date{}

\begin{document}
\maketitle

\begin{abstract}
 The Loewner energy of a Jordan curve is the Dirichlet energy of its Loewner driving term. 
 It is finite if and only if the curve is a Weil-Petersson quasicircle. 
 In this paper, we describe cutting and welding operations on finite Dirichlet energy functions defined in the plane, allowing expression of the Loewner energy in terms of Dirichlet energy dissipation. 
 We show that the Loewner energy of a unit vector field flow-line is equal to the Dirichlet energy of the harmonically extended winding. We also give an identity involving a complex-valued function of finite Dirichlet energy that expresses the welding and flow-line identities simultaneously.
 As applications, we prove that arclength isometric welding of two domains is sub-additive in the energy, and that the energy of equipotentials in a simply connected domain is monotone. Our main identities can be viewed as action functional analogs of both the welding and flow-line couplings of Schramm-Loewner evolution curves with the Gaussian free field.
\end{abstract}

\section{Introduction}
Let $\eta$ be a Jordan curve in $\Chat = \mathbb{C}\cup\{\infty\}$.
The Loewner equation describes such a curve by a real-valued continuous function on $\m R$ called the Loewner driving term. The M\"obius invariant \emph{Loewner energy} of $\eta$, denoted $I^L(\eta)$, is by definition the Dirichlet energy of this driving term \cite{RW,W2}.
It was shown in \cite{W2} that if $\eta$ passes through $\infty$, then we have the following equivalent expression:
 \begin{equation}\label{eq:feb8.1}
 I^L(\eta) = \frac{1}{\pi} \int_{\HH} \left|\nabla \log|f'| \right|^2 \, dz^2 + \frac{1}{\pi} \int_{\HH^*} \left|\nabla \log|g'| \right|^2 \, dz^2.
 \end{equation}
 Here $f$ and $g$ map conformally the upper and lower half-planes $\HH$ and $\HH^*$ onto, respectively, $H$ and $H^*$, the two components of $\mathbb{C} \smallsetminus \eta$, while fixing $\infty$. (Here and below $dz^2$ denotes two-dimensional Lebesgue measure.)
Moreover, a Jordan curve has finite energy if and only if it is a \emph{Weil-Petersson quasicircle}, that is, its normalized welding homeomorphism belongs to the Weil-Petersson  Teichm\"uller space \cite{W2}, which appears, e.g., in the context of closed string theory and has attracted considerable interest from both mathematicians and physicists, see, e.g., \cite{BR87a, NS,TT2006WP,Shen2013}. The link with the Loewner energy goes deeper and the energy itself is intimately connected to the geometry of the Weil-Petersson Teichm\"uller space: 
 it coincides with (a constant times) the universal Liouville action of Takhtajan and Teo \cite{TT2006WP}, a K\"ahler potential for the Weil-Petersson metric.

Another motivation to study the Loewner energy is that it is also the action functional of the Schramm-Loewner evolution (SLE), a family of random fractal curves arising as universal scaling limits of interfaces in critical planar lattice models. Pioneering work of Dub\'edat \cite{Dubedat_GFF} and Sheffield \cite{Sheffield_QZ} on \emph{couplings} between SLEs and the Gaussian free field (GFF) have led to remarkable and far-reaching results, see, e.g., \cite{Miller_Sheffield_IG1,mating-of-trees}. Our main identities are in a certain sense deterministic analogs of SLE/GFF coupling theorems, on the action functional level. We will further comment on this at the end of the introduction. One of the original motivations for this work was indeed to better understand the SLE/GFF relations. However, we stress that our (short) proofs use only analytic tools, and we need no results about the probabilistic models in this paper.
\subsection{Cutting and welding}

Our first theorem exhibits the close interplay between \field{}s of finite Dirichlet energy in the plane and the Loewner energy of a Jordan curve passing through $\infty$.
 To state the result, 
 we write $\mc{E}(\Omega)$ for the space of real functions on a domain $\Omega \subset \m C$ with weak first derivatives in $L^2(\Omega)$, and define  the Dirichlet energy of $\phi \in \mc E(\Omega)$ by 
 \[\mc D_\Omega(\phi) := \frac{1}{\pi} \int_\Omega |\nabla \phi|^2 dz^2.\]
 \begin{thm}[Cutting]\label{thm:welding_coupling1}
    Suppose $\eta$ is a Jordan curve through $\infty$, let $f$ and $g$ be conformal maps associated to $\eta$ as above, and suppose $\phi \in  \mc E(\mathbb{C})$ is given. Then we have the identity:
    \begin{equation}\label{jan26.1}
    \mc D_\mathbb{C}(\varphi) + I^L(\eta)  = \mc D_{\mathbb{H}}(u) + \mc D_{\mathbb{H}^*}(v) ,\end{equation}
    where  
    \begin{equation}\label{eq:ppS}
     u =  \varphi \circ f + \log \abs{f'} \text{ and } v =  \varphi \circ g + \log \abs{g'}.
\end{equation}

\end{thm}
It is natural to view the functions in Theorem~\ref{thm:welding_coupling1} as real parts of ``pre-pre-Schwarzian'' forms whose transformation law is given by \eqref{eq:ppS}.  Note that the Dirichlet energy is not invariant under this transformation.
It is  not hard to see that $e^{2 \phi} dz^2$ defines a locally finite measure on $\m{C}$, absolutely continuous with respect to Lebesgue measure $dz^2$. The transformation law \eqref{eq:ppS} shows that $e^{2u} dz^2$ and $e^{2v} dz^2$ are the pullback measures by $f$ and $g$ of $e^{2 \phi} dz^2$, respectively, see Section~\ref{sect:welding-coupling-proof}.

Theorem~\ref{thm:welding_coupling1} shows that a finite energy curve $\eta$ cuts a $\phi \in \mc E(\m C)$ into two half-plane forms in a way that conserves the total energy. 
Note also that when $\phi$ is constant, \eqref{jan26.1} reduces to the identity \eqref{eq:feb8.1}. See Theorem~\ref{thm_coupling_identity_H} for the proof of 
Theorem~\ref{thm:welding_coupling1} and Theorem~\ref{thm_coupling_identity_D} for the corresponding identity for a bounded Jordan curve. 
\begin{figure}[ht]
 \centering
 \includegraphics[width=0.7\textwidth]{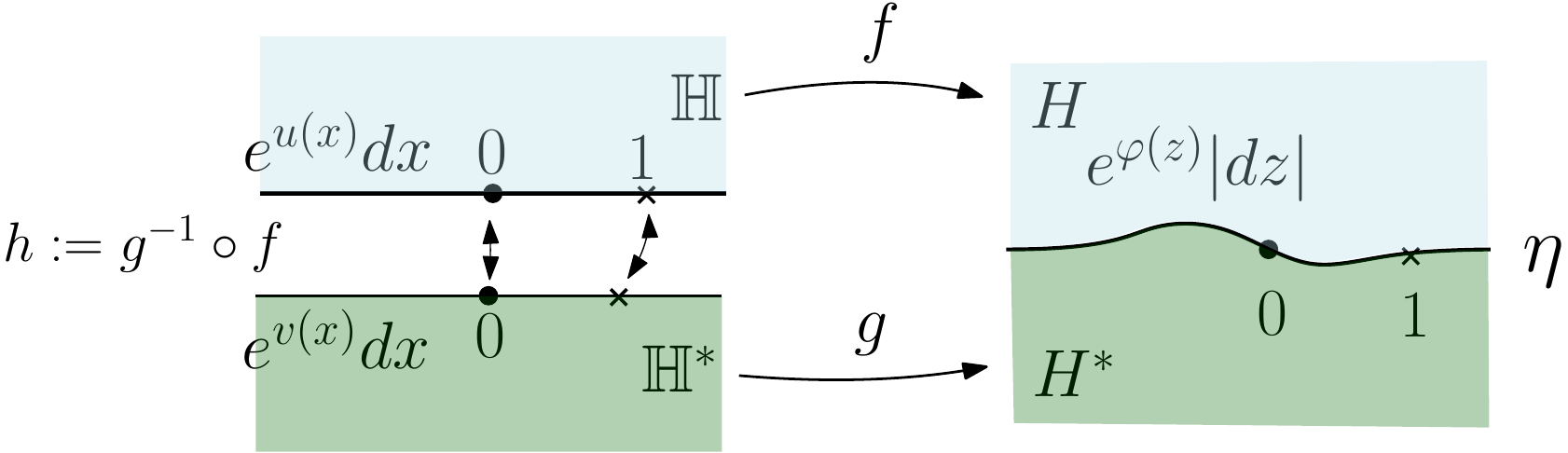}
 \caption{\label{hp_welding} Isometric conformal welding: $h = g^{-1} \circ f$ is constructed from the measures $e^u dx$ and $e^v dx$,  and their pushforward measures by $f$ and $g$  both give $e^{\phi} |dz|$ on $\eta$. } 
 \end{figure}

Given two half-plane \field{}s of finite Dirichlet energy, one can recover $\phi$ and $\eta$ such that \eqref{jan26.1} holds. 
In fact, the operation converse to cutting is implemented by \emph{conformal welding}: An increasing homeomorphism $h:  \m R \to  \m R$ is said to be a (conformal) welding homeomorphism if there is a Jordan curve $\eta$ through $\infty$ and conformal maps $f,g$ of the upper and lower half-planes onto the two components of $\m C \smallsetminus \eta$, respectively, such that $h = g^{-1} \circ f|_{\m R}$.

Suppose $\m H$ and $\m H^*$ are each equipped with a boundary measure defining a distance between $x < y$ by the measure of $[x,y]$. Under suitable assumptions on the measures, the isometry $h: \m R = \partial \m H \to \partial \m H^* = \m R$ fixing $0$ is well-defined and a welding homeomorphism. In this case, we say that $h$ is an \emph{isometric welding} homeomorphism and the corresponding tuple $(\eta, f,g)$ is a solution to the isometric welding problem for the given measures. 

In our setting we have the following result. See Theorem~\ref{thm:tuple_3_2} for a complete statement and Theorem~\ref{thm:tuple_D} for the welding of disks. 
\begin{thm}[Isometric conformal welding]\label{thm:tuple12}
Suppose $u \in \mc E(\mathbb{H})$ and  $v \in \mc E(\mathbb{H}^*)$ are given. The isometric welding problem for the measures $e^{u} dx$ and $e^{v} dx$ has a  solution $(\eta, f,g)$ and the welding curve $\eta$ has finite Loewner energy. Moreover, there exists a unique
$\varphi \in \mc E(\mathbb{C})$ such that \eqref{jan26.1} and \eqref{eq:ppS} are satisfied.
 \end{thm}
In the statement, $dx$ is Lebesgue measure on $\m{R}$ and the measures $e^{u} dx$ and $e^{v} dx$ are defined using the traces of $u, v$ on $\mathbb{R}$.
The solution $(\eta, f,g)$ in Theorem~\ref{thm:tuple12} is in fact unique if appropriately normalized, see Section~\ref{sec:conformal_welding}. 

We have the following consequence which justifies calling conformal welding the inverse to cutting, see
Corollary~\ref{cor:3_2_alternative} for the precise statement.
Let $u \in \mc E(\m H)$ and $v \in \mc E(\m H^*)$ be forms with transformation law \eqref{eq:ppS}, and assume they ``glue'' to a \field{} $\phi \in \mc E(\m C)$ along a Jordan curve.
Then the interface is necessarily obtained by the isometric welding of the boundary measures, and its Loewner energy is finite and given by the difference $\mc D_{\m C}(\phi) - \mc D_{\m H}(u)-\mc D_{\m H^*}(v)$.  In this way we may view the Loewner energy as quantifying the ``dissipation'' of Dirichlet energies when performing this gluing operation.

In order to prove Theorem~\ref{thm:tuple12}, we first show that the welding curve $\eta$ has finite energy, which implies $\vp \in \mc{E}(\m C \smallsetminus \eta)$. The idea is to then show that $\varphi$ has matching traces defined from both sides of $\eta$, and use this to conclude $\varphi \in \mc E(\m C)$. We may then apply Theorem~\ref{thm:welding_coupling1} to obtain \eqref{jan26.1}.

Although finite energy curves are not $C^1$ and may exhibit slow spirals \cite{RW} (and so are not Lipschitz), there is still enough regularity to take traces using disk averages of elements in $\mc E(\m C)$ (see Appendix~\ref{sect:trace}) giving rise to $H^{1/2}$ spaces along the curves and to employ BMO-space estimates. Conformal invariance properties and the good interplay between arclength and harmonic measure on finite energy curves are also important for the analysis. In relation to this, let us briefly explain a simple consequence for a question in geometric function theory. 

  Suppose $\eta_1,\eta_2$ are locally rectifiable Jordan curves of the same length (possibly infinite) bounding two domains $\O_1$ and $\O_2$ and mark a point on each curve. Let $\psi$ be an arclength isometry $\eta_1 \to \eta_2$ matching the marked points. Following Bishop \cite{bishop_isometric}, we are interested in whether there is a Jordan curve $\eta$ (and whether it is unique up to M\"obius transformation), and conformal equivalences $f_1, f_2$ from $\O_1$ and $\O_2$ to the two connected components of $\m C \smallsetminus \eta$, such that $f_2^{-1} \circ f_1|_{\eta_1} = \psi$, that is, we are asking whether $\psi$ is an arclength isometric welding.  
  Rectifiability of $\eta_1$ and $\eta_2$ does not guarantee the existence nor the uniqueness of $\eta$, but the chord-arc property does (see below for the definition). However, chord-arc curves are not closed under isometric conformal welding: the welding curve can have Hausdorff dimension arbitrarily close to $2$, see \cite{david_chord_arc, Semmes86, bishop_isometric}.
  We will show that finite energy curves behave much better.
    Theorem~\ref{thm:welding_coupling1} and Theorem~\ref{thm:tuple12} together imply the following result.

\begin{cor}\label{cor:isometric-welding}
The class of finite energy curves is closed under arclength isometric welding.
Moreover, if $\eta$ is the welding curve corresponding to the arclength isometric welding of $\eta_1$ and $\eta_2$, then 
\[
I^L(\eta) \le I^L(\eta_1) + I^L(\eta_2). 
\]
\end{cor}
\begin{rem} The inequality of Corollary~\ref{cor:isometric-welding} can be interpreted as an energy dissipation of the arclength isometric welding into an ambient \field{} in $\mc{E}(\m{C})$. See Corollary~\ref{cor:arc-length} for the precise (in fact stronger) statement for unbounded curves, and Corollary~\ref{cor:arc-length-D} for bounded curves.
\end{rem}

\subsection{Flow-line identity}

An elementary observation is that the Dirichlet energy of a harmonic function is equal to  the Dirichlet energy of its harmonic conjugate. Therefore \eqref{eq:feb8.1} can be written 
$$I^L(\eta) = \frac{1}{\pi} \int_{\HH} \left|\nabla \arg f' \right|^2 \, dz^2 + \frac{1}{\pi} \int_{\HH^*} \left|\nabla \arg g' \right|^2 \, dz^2.
 $$
On the other hand, if $\eta$ has finite energy, the boundary value of (a continuous branch of) $\arg f'\circ f^{-1}$ gives the winding $\tau  = \arg \eta'$ of $\eta$, and  (Theorem~\ref{thm_flow_line}) one can express the Loewner energy as
\begin{equation} \label{eq:intro_flow_tau}
I^L(\eta) = \mc D_{\m C \smallsetminus \eta} (\mc P[\tau]), 
\end{equation}
where we write $\mc{P}[\tau]$ for the harmonic extension of $\tau$ to $\m {C}\smallsetminus \eta$ on both sides, that is, the restrictions are given by the solutions to the Dirichlet problem with boundary data $\tau$. (Actually, we will show that $\mc P[\tau] \in \mc{E}(\m C)$ so \eqref{eq:intro_flow_tau} holds with
$\mc D_{\m C \smallsetminus \eta}$ replaced by $\mc D_{\m C}$.)

Consider a unit vector field $X(z) = e^{i \phi(z)}$ on $\m C$. A flow-line of $X$ through a point $z_0 \in \m C$ is a solution to the differential equation
    $$\dot \eta(t) = X(\eta(t)), \quad t \in (-\infty, \infty), \quad \eta(0) = z_0.$$ Equation~\eqref{eq:intro_flow_tau} can then be used to prove the following result, see Theorem~\ref{thm_flow_line}.

\begin{thm}[Flow-line identity]
Let $\varphi \in \mc E (\m C) \cap C^0(\hat{\m C})$. Any flow-line $\eta$ of the vector field $e^{i\varphi}$ is a Jordan curve through $\infty$ with finite Loewner energy and we have the formula
\[
\mc D_{\m C} (\varphi) = I^L(\eta) + \mc D_{\m C} (\varphi_0),
\]
where $\varphi_0 = \varphi - \mc P [\varphi|_{\eta}]$.
\end{thm}

Using these facts, we deduce that the energy of equipotentials is monotone, see Corollaries~\ref{cor_compare_horo} and~\ref{cor:compare_equipotential}.  We summarize these results below, see also Corollaries~\ref{cor_equi_convergence} and~\ref{cor_approx_H}. 
There are two different cases: $\eta$ is a bounded finite energy Jordan curve (resp. passing through $\infty$), and $f$ a conformal map from $\m D$ (resp. $\m H$) to one connected component of $ \m C \smallsetminus \eta$. 

\begin{cor}
Consider the family of analytic curves $\eta_r := f(r \m T)$, where $0< r <1$  (resp. $\eta^{r} := f(\m R + ir)$, where $r > 0$).
   For all $0 < s < r <1$ (resp. $0 < r < s$), we have
   $$ I^L(\eta_s) \le I^L(\eta_r) \le I^L(\eta), \quad (\text{resp. } I^L(\eta^s) \le I^L(\eta^r) \le I^L(\eta),)$$
   and equalities hold if only if $\eta$ is a circle (resp. $\eta$ is a line).
   Moreover, $I^L(\eta_r)$ (resp. $I^L(\eta^r)$) is continuous in $r$ and
   \begin{align*}
   &  I^L(\eta_r) \xrightarrow{r \to 1-} I^L(\eta); \quad I^L(\eta_r) \xrightarrow{r \to 0+} 0\\
   (\text{resp. } & I^L(\eta^r) \xrightarrow{r \to 0+} I^L(\eta); \quad I^L(\eta^r) \xrightarrow{r \to \infty} 0).
   \end{align*}
\end{cor}

\begin{rem}
Both limits and the monotonicity substantiate the intuition that the Loewner energy measures the deviation of a Jordan curve from a circle. In particular, the vanishing of the energy of $\eta_r$ as $r \to 0$ can be thought as expressing the fact that conformal maps asymptotically take small circles to circles. 
\end{rem}

We have the following corollary which expresses both welding and flow-line identities simultaneously, see Corollary~\ref{cor:complex_field} for the precise statement. 

\begin{cor}[Complex identity]
Let $\psi$ be a complex-valued function on $\m C$ with finite Dirichlet energy and imaginary part continuous in $\Chat$.
   Let $\eta$ be a flow-line of the vector field $e^{\psi}$ and 
   $f, g$ the conformal maps as in Figure~\ref{hp_welding}.
  Then we have
  $$ \mc D_{\m C}(\psi) = \mc D_{\m H}(\zeta) + \mc D_{\m H^*}(\xi),$$
  where $\zeta = \psi \circ f + (\log f')^*$, $\xi = \psi \circ g + (\log g')^*$ and $z^*$ is the complex conjugate of $z$.
\end{cor}

\subsection{SLE/GFF discussion and heuristics}\label{sec:dictionary}
As we have indicated, the relations given in our main theorems can be interpreted as action functional analogs of SLE/GFF coupling theorems. We will make some remarks related to this, but we emphasize that we are not making rigorous statements here.

Recall that the GFF is a Gaussian random distribution whose correlation function is given by the Green's function for the Laplacian and that SLE$_\kappa$ is the family of random curves obtained by using $\sqrt{\kappa}B_t$, where $B_t$ is standard Brownian motion, as driving function for the Loewner equation, see \cite{SheffieldGFF, BasicSLE}.

If $X$ is a centered Gaussian random variable with law $\mu$, taking values in a Banach space, 
the family of random variables $(\sqrt{\k} X)_{\k \to 0+}$ has large deviation rate function equal to the action functional (associated Cameron-Martin norm) for $\mu$.
For example, the large deviation rate function for the Neumann GFF on a domain $\Omega$ is $I_{\GFF} (\phi) =  \int_{\Omega} \abs{\nabla \phi (z)}^2 /4 \pi \, dz^2 = \mc D_\Omega (\phi)/4$.
 Moreover, since the one-dimensional Dirichlet energy is the action functional for Brownian motion, we expect SLE$_\kappa$ and the Loewner energy to be related as:
\begin{equation} \label{feb20.1}
 - \k \log \mathbb{P}\left\{ \SLE_\k \text{ loop stays close to } \eta \right\} \approx I^L(\eta), \qquad \kappa \to 0\!+\!.
 \end{equation}
  See \cite{W1} for a precise statement in the chordal setting.

   Sheffield's quantum zipper couples SLE$_\k$ curves with \emph{quantum surfaces} via a cutting operation and as welding curves\footnote{The welding homeomorphisms that arise in the random setting here are very rough and solving the associated welding problems directly in the analytic sense as in \cite{AJKS} is still an open problem. In \cite{Sheffield_QZ} the coupling is constructed using the reverse SLE flow, and the fact that it corresponds to isometric welding is checked \emph{a posteriori}.} \cite{Sheffield_QZ, mating-of-trees}. A quantum surface is a domain equipped with a
  Liouville quantum gravity ($\g$-LQG) measure, defined using a regularization of $e^{\gamma \Phi}dz^2$, where $\g = \sqrt \k \in (0,2)$, and $\Phi$ is a Gaussian field with the covariance of a Neumann GFF. 
 There is another coupling known as the forward SLE/GFF coupling, of critical importance, e.g., in the imaginary geometry framework of Miller-Sheffield \cite{Dubedat_GFF, Miller_Sheffield_IG1}: very loosely speaking,  an SLE$_\k$ curve may be coupled with a GFF $\Phi$ and thought of as a (measurable) flow-line of the vector field $e^{i \Phi/\chi}$, where $\chi = 2/\g - \g/2$.

Given these and similar observations, it is possible to guess our identities via heuristic large deviation arguments analogous to \eqref{feb20.1} in the small $\g$ limit. We start from the probabilistic coupling theorems and express the independence as summation of action functionals on the deterministic side. Note that the leading order log-singularities representing conical singularities in the relevant quantum surfaces vanish as $\gamma \to 0+$ and we have $\chi \sim 2/\g$.  Let us finally remark that the complex identity, Corollary~\ref{cor:complex_field}, which expresses both welding and flow-line identities simultaneously, is actually the finite energy analog of the mating of trees theorem of Duplantier, Miller, and Sheffield \cite{mating-of-trees}.
This analogy is not as apparent as in the other cases and details will appear elsewhere \cite{VW2}. The picture that emerges can be summarized in the following table. We will not go into further details here.
\vskip 0.3 cm
\begin{center}
\begin{tabular}{| l | l | }
  \hline			
  {\bf SLE/GFF}  & {\bf Finite energy}  \\ \hline 
      SLE$_\kappa$ loop. & Finite energy Jordan curve, $\eta$. \\ \hline
       $\g$ times Neumann GFF $\g \Phi$ on $\mathbb{H}$ (on $\m C$). &  $2u, \, u \in \mathcal{E}(\HH)$ ($2\phi, \, \phi \in \mathcal{E}(\m C)$).   \\ \hline
              $\g$-LQG on quantum plane $\approx e^{\g \Phi} dz^2$. & $e^{2 \varphi(z)} dz^2, \, \varphi \in \mathcal{E}(\m C)$. \\ \hline

  $\gamma$-LQG on quantum half-plane on $\mathbb{H}$ & $e^{2 u(z)} dz^2,  \, u \in \mathcal{E}(\HH)$.
  \\ \hline
  $\gamma$-LQG boundary measure on $\mathbb{R}$ $\approx e^{\g \Phi/2} dx$ & $e^{u(x)}dx, \, u \in H^{1/2}(\mathbb{R})$. \\ \hline
       Independent SLE$_\kappa$ cuts a   & Finite energy $\eta$ cuts $\phi \in \mathcal{E}(\mathbb{C})$\\
     quantum plane into   & into $u \in \mc E(\m H), v \in \mc E(\m H^*)$ and \\
     independent quantum half-planes. &  $I^L(\eta) + \mc{D}_{\m{C}}(\phi) = \mc{D}_{\m{H}}(u) + \mc{D}_{\m{H}^*}(v).$
     \\
     \hline
    Isometric welding  & Isometric welding \\ 
    of independent $\gamma$-LQG measures on $\mathbb{R}$ & of $e^{u} dx$ and $e^{v} dx$, $u, v \in H^{1/2}(\m R)$\\
        produces SLE$_\kappa$. & produces a finite energy curve. \\
 \hline
    $\gamma$-LQG chaos  w.r.t. Minkowski content & $e^{\vp|_\eta} |dz|$,  $\vp|_\eta \in H^{1/2}(\eta),$ \\  
      equals the pushforward of & equals the pushforward of \\
      $\g$-LQG measures on $\m R$. &  $e^u dx$ and $e^v dx$, $u,v \in H^{1/2}(\m R)$. \\ 
     \hline

     Bi-infinite flow-line of $e^{i \Phi / \chi }\approx e^{i \g \Phi / 2 }$  &  Bi-infinite flow-line of $e^{i \varphi}$   \\ 
      is an SLE$_\kappa$ loop. &  is a finite energy curve. \\
  \hline  
     Mating of trees  &  Complex identity \\ 
  \hline  
\end{tabular}
\end{center}

\bigskip
\noindent \textbf{Conventions: }
Throughout the paper, we consider implicitly all Jordan curves to be oriented, so that the complement has two connected components denoted $H$ and $H^*$  ($\O$ and $\O^*$) when the curve is unbounded (bounded), where the curve winds counterclockwise around $H$ and $\O$, and clockwise around $H^*$ and $\O^*$. We also choose the orientation for bounded curves so that $\O$ is the bounded component. 

\bigskip
{\bf Acknowledgements: } F.V. acknowledges support from the Knut and Alice Wallenberg foundation and the Swedish Research Council. Y.W. is partially supported by the Swiss National Science Foundation grant \# 175505. Part of this work was carried out at IPAM/UCLA, Los Angeles. It is our pleasure to thank Alexis Michelat for the proof of Lemma~\ref{lem:density}, Michael Benedicks and Scott Sheffield for discussions, and Juhan Aru, Wendelin Werner, and the referees for very helpful comments on earlier versions of our paper. We are also happy to thank Juhan Aru for asking a question that led us to Corollary~\ref{cor:complex_field}.

\section{Preliminaries}\label{sect:prel}
\begin{spacing}{1.1}
For an open, connected set $\Omega \subset \hat{\mathbb{C}}$, we write $W^{1,2}(\Omega)$ for the Sobolev space of real-valued functions $u$ such that both $u$ and its weak first derivatives are in $L^2(\Omega)$.
We use the norm $\|u\|_{W^{1,2}(\Omega)} = \|u\|_{L^2(\Omega)} + \|\nabla u\|_{L^2(\Omega)}$ and denote by $W^{1,2}_0(\Omega)$ the closure of $C_c^\infty(\Omega)$ in $W^{1,2}(\Omega)$. 

Let $\mathcal{E}(\Omega)$ be the homogeneous Sobolev space of functions on $\Omega$ with 
finite Dirichlet energy, which differs from $W^{1,2}(\Omega)$ when $\Omega$ is unbounded.
Note that $C^{\infty}_c (\m C)$ is dense in $\mc E(\m C)$ with respect to the semi-norm $\mc D_{\m C}(\cdot)^{1/2}$, see Lemma~\ref{lem:density}. 
Since the Dirichlet energy is conformally invariant so are the $\mc{E}(\Omega)$ spaces: if $f : \O_1 \to \O_2$ is conformal, then for all $u \in \mc E(\O_2)$, $\mc D_{\O_1} (u \circ f) = \mc D_{\O_2} (u)$.

Let $\mc E_{\text{harm}}(\O) \subset \mathcal{E}(\O)$ be the conformally invariant space of  harmonic functions on $\O$ with finite Dirichlet energy. Recall that $\mc E_{\text{harm}} (\m C)$ only consists of constant functions.
\end{spacing}

\begin{lemma}[{\cite[p.77]{Adams}}]\label{lem:decomp}  
Assume that $\partial \O$ is non-polar.
 For $u \in W^{1,2}(\O)$, we have the unique decomposition
   $$u = u_0 + u_{h}$$
   where $u_0 \in W^{1,2}_0(\Omega)$ and $u_{h} \in \mc E_{\text{harm}} (\O)$.
\end{lemma}
We will use the following form of the Poincar\'e inequality which can be proved using a scaling argument: 
if $D \subset \mathbb{C}$ is a disk or square and $u\in W^{1,2}(D)$, then $$\int_{D}|u-u_{D}|^2 dz^2 \le 4 \, (\diam \, D)^2 \int_{D}|\nabla u|^2 dz^2.$$ Here and in the sequel we use the notation
 \[ u_\Omega = \frac{1}{|\Omega|}\int_\Omega u \, dz^2
 \]
 for the average of $u \in L^1(\Omega)$ over $\Omega$ and we write $|\Omega|$ for the Lebesgue measure of $\Omega$.

A quasidisk is a simply connected domain whose boundary is a quasicircle in $\hat{\mathbb{C}}$, that is, the image of the unit circle or real line under a global quasiconformal homeomorphism of $\m C$. Quasidisks are extension domains for both $W^{1,2}$ and $\mc E$; see, e.g., Theorem~1 and the explicit quasiconformal reflection on p.72 of \cite{Jones}: 
\begin{lemma}\label{lem:jones-extension}
   Suppose $\Omega$ is a quasidisk and that $u\in W^{1,2}(\Omega)$ and $v \in \mathcal{E}(\Omega)$. Then $u$ extends to a function $\tilde u \in W^{1,2}(\mathbb{C})$ such that $\tilde{u}|_{\Omega} = u$ and $v$ extends to a function $\tilde{v} \in \mathcal{E}(\mathbb{C})$ such that $\tilde{v}|_{\Omega} = v$.
\end{lemma}

A function $u: \mathbb{C} \to \mathbb{R}$ is said to have bounded mean oscillation, $u \in$ BMO, if 
\[
\sup_{D}\frac{1}{|D|}\int_D|u - u_D| \, dz^2 =: \|u\|_{*} < \infty,  
\]
where $D$ ranges over squares with sides parallel to the axes. The definitions for the BMO spaces of functions on $\mathbb{R}$ or $\mathbb{T}$ are the same, replacing squares by intervals and considering the appropriate Lebesgue measure. 
We have the John-Nirenberg inequality, see Theorem~VI.6.4 of \cite{garnett}:
There exist $c_0,C$ such that for every $u \in$ BMO and every cube~$D$,
\begin{equation}\label{JN1}
|\{x: |u(x) - u_D| > \lambda \}|/|D| \le C  \exp\left(-c_0 \lambda /\|u\|_{*}\right). 
\end{equation}

\subsection{$H^{1/2}$ on chord-arc curves}
We will see below that finite energy curves are chord-arc, that is, they are Ahlfors regular quasicircles: there is a constant $K$ so that for every $x,y $ on the curve $\eta$, the shorter arc $\eta_{x,y}$ between $x$ and $y$ satisfies the estimate $$\textrm{length}\left(\eta_{x,y}\right) \le \,  K |x-y|.$$
(Recall that $\eta$ is a quasicircle if and only if the same estimate holds, with length replaced by diameter.)
Equivalently, a chord-arc curve is the image of $\m T$ or $\m R$ under a bi-Lipschitz homeomorphism of $\mathbb{C}$, see \cite{jerison-kenig82}. It is easy to see that a curve through $\infty$ is chord-arc if and only if it is the M\"obius image of a bounded chord-arc curve. 

Suppose $\eta$ is a chord-arc curve in $\hat{\mathbb{C}}$. We define the
 homogeneous Sobolev space 
\[
H^{1/2}(\eta)=\left\{u: \eta \to \mathbb{R}: \|u\|_{H^{1/2}(\eta)} <\infty \right\},\]
where
\[\|u\|_{H^{1/2}(\eta)}^2:=\frac{1}{2\pi^2}\iint_{\eta \times \eta} \frac{|u(z) - u(w)|^2}{|z-w|^2} |dz||dw|
\]
 defines a semi-norm on $H^{1/2}(\eta)$. 
 Since the measure $1/\abs{z - w}^2 \abs{dz} \abs{dw}$ is invariant under any M\"obius transformation $m$ of $\Chat$, we have 
\begin{equation} \label{eq:norm_1/2_mobius_invariant}
 \norm{u}_{H^{1/2} (\eta)} = \norm{ u \circ m^{-1}}_{H^{1/2} (m (\eta))}.
\end{equation}
In fact, the $H^{1/2}$ space is also invariant under conformal mapping to another chord-arc domain (see also Lemma~\ref{lem:trace_commutation}): if $\eta$ is chord-arc and bounds the domain $\O$ and $\varphi: \mathbb{D} \to \Omega$ is some choice of Riemann map, then $u \circ \varphi \in H^{1/2}(\m T)$. Indeed, by M\"obius invariance we may assume $\eta$ is bounded in $\m C$. 
   Then there exists a bi-Lipschitz homeomorphism of the plane $\psi$ such that $\eta=\psi(\m T)$ and it easy to check that $u \circ \psi \in H^{1/2}(\m T)$. We may extend $\varphi$ to a quasiconformal map of the whole plane and therefore $\psi^{-1} \circ \varphi|_{\m T}$  is a quasisymmetric homeomorphism of $\m T$. Since  $u \circ \varphi =  u \circ \psi \circ (\psi^{-1} \circ \varphi)$, the next lemma shows that $u \circ \varphi \in H^{1/2}(\m T)$.

 \begin{lemma}\label{lem:qsh}
If $u\in H^{1/2}(\mathbb{R})$ and $h:\mathbb{R} \to  \mathbb{R}$ is a quasisymmetric homeomorphism, then $u \circ h \in H^{1/2}(\mathbb{R})$. 
Similarly, if $u\in H^{1/2}(\m T)$ and $h:\mathbb T \to  \m T$ is a quasisymmetric homeomorphism, then $u \circ h \in H^{1/2}(\m T)$. 
\end{lemma}
 See \cite[Section 3]{NS} for a proof of Lemma~\ref{lem:qsh} in the setting of the unit circle and the proof for the line is the same.

Any element of $\mc E_{\text{harm}}(\mathbb{D})$ has non-tangential limits almost everywhere on $\mathbb{T}$ (actually outside a set of zero logarithmic capacity).  
The functions on $\mathbb{T}$ obtained in this way coincides with the elements of  $H^{1/2} (\m T)$.  
The Dirichlet energy of the Poisson integral of a function in $H^{1/2}(\m T)$ can be expressed using Douglas' formula, see e.g. \cite{Ahlfors1}: If $u \in H^{1/2}(\m T)$ and $\mc P_{\mathbb{D}}[u]$ is its harmonic extension into $\m D$, then
    \begin{equation} \label{eq:doug_1}
         \mc D_{\m D} (\mc P_{\mathbb{D}}[u])   = \|u \|_{H^{1/2}(\mathbb{T})}^2
    \end{equation}
       and by \eqref{eq:norm_1/2_mobius_invariant} the analogous formula holds for $\mc E_{\text{harm}}(\mathbb{H})$.

If $u \in H^{1/2}(\mathbb{R})$, then $u$ has vanishing mean oscillation, $u \in \textrm{VMO}(\mathbb{R}) \subset \textrm{BMO}(\mathbb{R})$, that is, $u$ satisfies $\lim_{\delta \downarrow 0} \sup_{I: |I| \le \delta} \int_{I}|u-u_{I}|/|I|\,dx = 0$. (The VMO$(\m T)$ space is defined analogously.)
To see this, let $I$ be any bounded interval and set $u_I = \int_I u/|I| \, dx$. Then,
\begin{align*}
\frac{1}{|I|}\int_{I}|u-u_{I}|dx  \le \frac{1}{|I|^{2}} \iint_{I \times I}|u(x)-u(y)|dxdy \le  c \left(\iint_{I \times I}\frac{|u(x)-u(y)|^{2}}{|x-y|^2}dxdy \right)^{1/2}.
\end{align*}

For $\delta > 0$, write 
$$ \|u\|_{\delta, *} :=\sup_{I:|I| \le \delta}\frac{1}{|I|} \int_I |u-u_I| dx \le \|u\|_*.$$
Then by the John-Nirenberg inequality \eqref{JN1}, there exist $c_0,C$ 
such that for every $\delta > 0$, if $|I| \le \delta$, then
$\left|\{x: |u(x) - u_I| > \lambda \}\right|\le C  |I| \exp\left(-c_0 \lambda/\|u\|_{\delta, *}\right).$ 
It follows that if $u \in H^{1/2}(\mathbb{R})$ then we have $e^{u} \in L^1_{\textrm{loc}}(dx)$ by choosing $\d$ so that $\norm{u}_{\d, *} < c_0$. Moreover, there exists $c = c(\norm{u}_{\d, *})$ such that
\begin{equation}\label{eq:upper_lower_bound_integral}
 e^{u_I}  \le \frac{1}{|I|}\int_I e^{u(x)} dx \le c e^{u_I}   
\end{equation}
and $c \to 1$ as $\norm{u}_{\d, *} \to 0$.

Suppose $u \in \mc E(\m C)$ and that $\eta$ is a chord-arc curve in $\hat{\mathbb{C}}$. It is possible to define a trace of $u$ on $\eta$ by taking averages, and this trace will lie in $H^{1/2}(\eta)$. See Appendix~\ref{sect:trace} for more details. For now we just recall the definition: the limit
\begin{equation}\label{def:trace0}
\mc R_{\eta}[u](z):=\lim_{r \to 0+}u_{B(z,r)}, 
\end{equation}
where $B(z,r) = \{w: |w-z| < r\}$, exists for arclength a.e. $z \in \eta$ and $\mc R_{\eta}[u] \in H^{1/2}(\eta)$. 
The trace can also be defined from one side of the curve (Lemma~\ref{lem:trace_extension}): suppose $\partial \O = \eta$ and let  $u \in \mc E (\O)$.
Then
\begin{equation*}
\mc R_{\O \to \eta}[u](z) := \mc R_{ \eta}[\tilde u] (z) ,\quad \text{ for arclength a.e. }z\in \eta,
\end{equation*}
  is independent of any choice of 
  $\tilde u \in \mc E (\m C)$ 
  such that $\tilde u|_{\O} =u$.

\subsection{Conformal welding}\label{sec:conformal_welding}

Let $h$ be an increasing homeomorphism of $\m R$. We say that the triple $(\eta, f, g)$ is a  \emph{normalized solution to the conformal welding problem for $h$} if
\begin{itemize}[itemsep= -1 pt, topsep= -1pt]
\item $\eta$ is Jordan curve in $\Chat$ passing through $0,1, \infty$;
    \item $f: \m H \to H$ is the conformal map fixing $0, 1, \infty$;
    \item $g: \m H^* \to H^*$ is conformal and $g^{-1} \circ f = h$ on $\m R$,
\end{itemize}
where $H, H^*$ are the connected components of $\Chat \smallsetminus \eta$. 
It is well-known that if $h$ is quasisymmetric, then the solution is unique and $\eta$ is a quasicircle in $\hat{\mathbb{C}}$.

 Let $u,v \in H^{1/2}(\mathbb{R})$ be given and define  measures $d\mu=e^{u} dx$ and $d \nu = e^{v} dx$. 
  We define an increasing homeomorphism $h$ by $h(0)=0$ and then
\begin{equation}\label{def:homeo1}
h(x)=
\begin{cases}
      \inf\left\{ y\ge 0: \mu[0,x] = \nu[0,y]\right\} & \text{ if } x > 0;\\
    -\inf\left\{ y\ge 0: \mu[x,0] = \nu[-y,0]\right\} & \text{ if } x <0.
   \end{cases} \end{equation}
Using Lemma~\ref{lem:infinite_measure} below, $h$ is well-defined and $\mu([a,b]) = \nu(h([a,b]))$ for any choice of $a \le b$.  Proposition~\ref{prop:isometricwelding} will show that $h$ is quasisymmetric. Hence $h$ is the \emph{isometric welding homeomorphism} associated with $\mu$ and $\nu$, see Figure~\ref{hp_welding}.

  \begin{lemma} \label{lem:infinite_measure}
  Suppose $u \in H^{1/2}(\mathbb{R})$ and $d \mu = e^{u} dx$. 
      Then $\mu(I) = \infty$ for any unbounded interval $I$.
   \end{lemma}

  \begin{proof}
Changing coordinates to $\mathbb{T}$ using the M\"obius invariance of $H^{1/2}$,
it is enough to show that if $u \in H^{1/2}(\mathbb{T})$, then $\exp(u(\theta) - 2\log |\theta|) d \theta$ has infinite integral on $J=[0,\delta]$ for $\d>0$.
     We know that $u \in \textrm{VMO}(\mathbb{T})$ and this is in fact enough to assume. By Lemma~VI.1.1 of \cite{garnett} the BMO property implies there is a constant $c$ such that if $|I| \subset |J|$ with $|I| < |J|/2$ then 
     $$|u_I - u_J| \le c \log(|J|/|I|)\|u\|_{\d,*}$$ and since $u \in$ VMO$(\m T)$ we may assume $\delta$ is so small that $|u_I - u_J| \le  \log(|J|/|I|)/2$ whenever $I \subset J$ is as above.  
     Set $I_k = (\delta \exp(-k-1),\delta \exp(-k)]$ and let $u_k$ be the average of $u$ on $I_k$. Then $|u_k| \le c + |u_J| + k/2 \le c' + k/2$ where $c'$ depends only on $u_J$ but $J$ is fixed from now on. Let $E_k = \{\theta \in I_k : |u-u_k| \le  k/2\}$. Then the John-Nirenberg inequality \eqref{JN1} implies 
     $$|E_k| \ge \bigg(1 - C \exp\Big(-\frac{c_0 }{2\|u\|_{\d, *} }k \Big)\bigg) |I_k| \gtrsim \exp(-k),$$
     where $\gtrsim$ means inequality with a multiplicative constant (which does not depend on $k$).
     It now follows that
     \begin{align*}
       \int_{J} \exp(u - 2\log |\theta|) \, d\theta & \ge  \sum_k \exp(-|u_k|)\int_{I_k} \exp(-|u-u_k| - 2\log |\theta|) \, d\theta  \\
       & \gtrsim \sum_k \exp(-|u_k| +2k) \exp(- k/2)|E_k| \\
       & \gtrsim \sum_k \exp(-k/2 + 2k - k/2 - k)
     \end{align*}
     which is infinite as claimed.
  \end{proof}

\begin{prop}\label{prop:isometricwelding}
The function $h$ defined as in \eqref{def:homeo1} is an increasing quasisymmetric homeomorphism of $\mathbb{R}$ such that $\log h'\in H^{1/2}(\mathbb{R})$.

\end{prop}
    \begin{proof} It follows from \eqref{eq:upper_lower_bound_integral} that $e^u dx$, $e^v dx$ are measures mutually absolutely continuous with respect to Lebesgue measure. 
    Lemma~\ref{lem:infinite_measure} then implies that $h_u (x) := \int_0^x e^{u}$ and $h_v (x) := \int_0^x e^{v}$ are increasing homeomorphisms of $\m R$.
By construction, $\log h_u', \log h_v' \in H^{1/2}(\mathbb{R})$, and by Theorem~\ref{thm:h12characterization} below,
    $h_u,h_v$ are Weil-Petersson class homeomorphisms and are in particular quasisymmetric.
Therefore $h := h_v^{-1} \circ h_u$ is also quasisymmetric and
\[
\log h' = -(\log h_v') \circ h + \log h_u'.
\]
Lemma~\ref{lem:qsh} implies $\log h' \in H^{1/2}(\mathbb{R})$.
\end{proof}

\subsection{Loewner energy}\label{sect:Loewner}

There are several different characterizations of finite Loewner energy curves; we collect here those that are relevant for this paper. It is convenient to start with the bounded case. See \cite{Cui2000,Shen2013,TT2006WP,W2} for proofs of the results summarized in the next theorem.

      Given a Jordan curve $\eta$ in $\m C$, let $\Omega$ and $\Omega^*$ be the connected components of $\CC \smallsetminus \eta$, $f$ a conformal map from $\DD$ onto $\Omega$, and $g$ a conformal map from $\DD^*$ onto $\Omega^*$ fixing~$\infty$. The welding homeomorphism of $\eta$ is $h := g^{-1} \circ f$ restricted to $\m T$. 
      
\begin{thm}[Bounded finite energy domains] \label{thm_TT_equiv_T01}
The following statements are equivalent:
\begin{enumerate}[itemsep= -1pt, topsep= -3pt]
\item $I^L (\eta) < \infty$;
\item $\norm{\log \abs{f'}}_{H^{1/2} (\m T)}^2 = \mc D_{\m D} (\log \abs{f'}) =  \int_{\DD} \left|f''(z)/f'(z)\right|^2  dz^2 / \pi  < \infty;$
\item $\norm{\log \abs{g'}}_{H^{1/2} (\m T)}^2 =  \mc D_{\m D^*} (\log \abs{g'})  <\infty$;
  \item \label{item:welding_S1} The welding homeomorphism $h$ is absolutely continuous and $\log \abs{h'} \in H^{1/2}(\mathbb{T})$. 
\end{enumerate} 
   Moreover,
    \begin{equation} \label{eq_disk_energy}
   I^L(\eta) =  \mc D_{\m D} (\log \abs{f'}) + \mc D_{\m D^*} (\log \abs{g'})+4 \log \abs{f'(0)} - 4 \log \abs{g'(\infty)},
 \end{equation}
where $g'(\infty):=\lim_{z\to \infty} g'(z) = \tilde g'(0)^{-1}$ and $\tilde g(z) := 1/g(1/z)$.
\end{thm}
The curves satisfying the above conditions are also known as \emph{Weil-Petersson quasicircles}. The right-hand side of \eqref{eq_disk_energy}, first considered in \cite{TT2006WP}, is called \emph{universal Liouville action}. 
 This quantity depends only on the equivalence class of quasisymmetric welding homeomorphisms of the circle modulo left M\"obius transformations, which is a model of universal Teichm\"uller space. The universal Liouville action is a K\"ahler potential of the Weil-Petersson metric on the Weil-Petersson Teichm\"uller space which consists of the (equivalence classes of) welding homeomorphisms satisfying Condition~\ref{item:welding_S1} of Theorem~\ref{thm_TT_equiv_T01}. We will refer to the set of such homeomorphisms as the Weil-Petersson class of homeomorphisms.
See \cite{TT2006WP} for background and more details.

Let us make some quick remarks about the geometry of finite energy curves. For simplicity, we assume $\eta$ is bounded.  Let $f: \DD \to \Omega$ be a conformal map as in Theorem~\ref{thm_TT_equiv_T01}. 
Since 
$\mc D_{\DD}(\log|f'|) < \infty$,
  $\log f'$ belongs to VMOA, the space of functions in the Hardy space $\mc  H^2$ with vanishing mean oscillation. A theorem of Pommerenke \cite{Pommerenke_VMOA} therefore shows that finite energy curves are \emph{asymptotically smooth}, that is, chord-arc with local constant $1$: for all $x,y$ on the curve, the shorter arc $\eta_{x,y}$ between $x$ and $y$ satisfies $$\lim_{|x-y| \to 0} \textrm{length}\, (\eta_{x,y})/|x-y| = 1.$$ Finite energy curves are not $C^1$ however, and may, e.g., exhibit slow spirals, and $C^{3/2-\vare}$ does not imply finite energy,
 see \cite{RW}.
 Since $\eta$ is rectifiable, $f' \in \mc H^1$ (the Hardy space) and has non-tangential limits a.e. on $\m T$.

Next, we record a continuity property of the Loewner energy/universal Liouville action which is implied by the $L^2$-convergence of the pre-Schwarzian of the conformal maps. 
To state the lemma, 
suppose $\eta$ has finite energy and let $\{\eta_n\}_{n= 1}^{\infty}$ be a family of bounded finite energy curves, and let $f_n$ and $g_n$ be a choice of conformal maps associated to $\eta_n$ as above, and $f$ and $g$ associated to $\eta$.

\begin{lemma}[\cite{TT2006WP} Corollary~A.4. and Corollary~A.6.]
\label{lem_TT_cor} 
If $f''_n/f'_n \to f''/f'$ in $L^2(\mathbb{D})$,
 then we have
 $$ \lim_{n\to \infty} I^L (\eta_n) =   I^L (\eta).$$
\end{lemma}

In fact, in this case, the welding homeomorphism of $\eta_n$ converges to the welding homeomorphism of $\eta$ with respect to the Weil-Petersson metric. 
We have the following immediate corollary.
\begin{cor}\label{cor_equi_convergence}
If $\eta$ is bounded and  $I^L(\eta) <\infty$, we have
 $$\lim_{r \to 1^-} I^L(\eta_r) = I^L(\eta) \quad \text{and}  \quad \lim_{r \to 0^+} I^L(\eta_r)  =  0,$$
 and $r \mapsto I^L(\eta_r)$ is continuous,  where  $(\eta_r := f(r \m T))_{0< r<1}$ are equipotentials of the domain bounded by $\eta$.
\end{cor}
\begin{proof}

It suffices to apply Lemma~\ref{lem_TT_cor} with $f_r (z) :  = f(r z)$. Using Carath\'eodory's kernel theorem and the $L^2$-integrability of $f''/f'$, it is easy to see that $f_r''/f_r'$ converges to $f''/f'$ in $L^2(\m D)$ as $r \to 1-$ 
(see \cite{W3}, proof of Theorem 5.1).
The continuity of $r \mapsto I^L(\eta_r)$ follows similarly.
Since $f''/f' \in L^2(\m D)$, we  also have
$$\lim_{r \to 0+}\int_{\m D} |f_r''(z)/f_r'(z)|^2 dz^2 = \lim_{r \to 0+} \int_{\abs{z} \le r} |f''(z)/f'(z)|^2 dz^2 =0$$
which then implies that the energy vanishes as $r \to 0+$.
\end{proof}
\begin{rem}
  As a consequence of the flow-line identity, we will see that $r \mapsto I^L(\eta_r)$ is non-decreasing, see Corollary~\ref{cor_compare_horo}.
\end{rem}

Now assume that $\eta$ is a Jordan curve through $\infty$,
and $H$ and $H^*$ are the two connected components of $\m C \smallsetminus \eta$. 
Let $f: \m H  \to H$ and $g: \m H^* \to H^*$ be conformal maps fixing $\infty$.
 The Loewner energy of $\eta$ can be expressed in terms of $f$ and $g$ as (see \cite{W2} Theorem 6.1)
\begin{align}
\label{eq_infinite_loewner}
\begin{split}
 I^L (\eta) 
  = \mc D_{\m H} (\log \abs{f'}) + \mc D_{\m H^*} (\log \abs{g'}).
 \end{split}
 \end{align}
Moreover, $I^L(\eta) < \infty$ if and only if $\mc D_{\m H} (\log \abs{f'}) < \infty$.

There is also a characterization of finite energy curves in terms of the welding homeomorphism $g^{-1}\circ f$ on the real line:
 \begin{thm}[{\cite{Shen-Tang,Shen2018}}]\label{thm:h12characterization}
  An increasing homeomorphism $h$ of $\m R$ is in the Weil-Petersson class if and only if $h$ is absolutely continuous and $\log h' \in H^{1/2} (\m R)$.
 \end{thm}
 The ``only if'' part of the characterization also follows easily from \eqref{eq_infinite_loewner}. In fact, if $I^L(\eta) < \infty$, then the trace of $\log \abs{f'}$ and $\log \abs{g'}$ are in $H^{1/2} (\m R)$.
Since $\eta$ is a quasicircle, its welding homeomorphism $h$ is quasisymmetric. 
Hence, a.e.,
$$\log h' = \log \left|(g^{-1} \circ f)' \right| = \log \abs{f'}  - \log \abs{g'\circ h} \in H^{1/2} (\m R).$$
(We may differentiate (a.e) since $\eta$ is chord-arc.)
In the last step we used Lemma~\ref{lem:qsh}.

We have also the convergence of the energy of equipotentials in the half-plane setting.  
\begin{cor}\label{cor_approx_H} 
Let $\eta^y : = f(\m R+ iy)$ for $y > 0$.
 Then we have
  $$\lim_{y \to 0+} I^L(\eta^y) = I^L (\eta),  \quad \lim_{y \to \infty} I^L(\eta^y)  =  0$$
  and $y \mapsto I^L(\eta^y)$ is continuous.
\end{cor}

\begin{proof}
   Choose $m$ to be a M\"obius transformation of $\Chat$ such that $m(\m H) = \m D$ and such that $\tilde D:=m(H)$ is also bounded. 
We have that $m(\eta^y)$ is the image of the horocycle (a circle tangent to $\m T$ at $m(\infty)$) $m(\m R +iy)$ by the conformal map $\tilde f  =  m \circ f \circ m^{-1}: \m D \to \tilde{D}$.
   Since $\partial \tilde{D} = \tilde f (\m T) = m(\eta)$ is bounded, the M\"obius invariance of the Loewner energy and the same proof as in Corollary~\ref{cor_equi_convergence} applied to $m(\eta^y)$ readily show that
$$\lim_{y \to 0+} I^L(\eta^y) = \lim_{y \to 0+} I^L(m(\eta^y)) = I^L (m (\eta)) = I^L(\eta),$$
and the continuity and $\lim_{y \to \infty} I^L(\eta^y)  =  0$ follow similarly.
\end{proof}

\section{Proofs of main results}

\subsection{Welding identity: half-plane version} \label{sect:welding-coupling-proof}
In this section, we prove the welding identity in the half-plane setting, and all curves are assumed to pass through $\infty$, see Section~\ref{sec:D_welding} where we discuss the analogous results in the finite case. 

If $\varphi \in \mc E (\m C)$, we have $e^{2 \phi} \in L^1_{loc}(dz^2)$.
To see this, let $D$ be a square with sides parallel to the axes. By the Cauchy-Schwarz and Poincar\'e inequalities, there is a universal constant $C$ such that
\[
\frac{1}{|D|}\int_D|\varphi-\varphi_D| dz^2  \le C \left(\int_D|\nabla \varphi|^2 dz^2\right)^{1/2}.
\]
The last integral is uniformly bounded in $D$ and tends to $0$ with $|D|$. The John-Nirenberg inequality \eqref{JN1} therefore shows that $e^{2\varphi} \in L^1_{\textrm{loc}}(dz^2)$, as claimed. Given a Jordan curve $\eta$, let $f,g$ be conformal maps from $\mathbb{H},\mathbb{H}^*$ onto $H, H^*$ fixing $\infty$, respectively. 
Next, define 
\begin{equation}\label{eq:def_u_v}
 u(z) = \varphi \circ f (z) + \log \abs{f'(z)}, \quad v(z) =  \varphi \circ g (z) + \log \abs{g'(z)}.
 \end{equation}
Then $e^{2 u} dz^2$ is the pullback of the measure $e^{2\varphi} dz^2 $ by $f$ on $\m H$,
   and  $e^{2 v} d z^2$ is the pullback of $e^{2\varphi} dz^2$ by $g$ on $\m H^*$.

If $\psi: \Omega \to \m C$ is an analytic function, we will use the shorthand notation
   $$ \s_\psi (z) : = \log \abs{\psi'(z)}, \quad  z \in \Omega. $$
    
 \begin{thm}\label{thm_coupling_identity_H} Let $\varphi \in \mc E (\m C)$ and let $\eta$ be a Jordan curve through $\infty$.
  We have the formula
   \begin{equation}
   \label{eq_coupling_H}
  \mc D_{\m C}(\varphi) + I^L(\eta) =\mc D_{\m H}(u) + \mc D_{\m H^*}(v).
   \end{equation}
\end{thm}
 \begin{proof}
 
Since $\varphi \in \mc E (\m C)$, by conformal invariance of the Dirichlet energy, $\mc D_{\m H} (u) =\infty$ if and only if $\mc D_{\m H} (\s_f) = \infty$ which is equivalent to $I^L(\eta) = \infty$.  The identity \eqref{eq_coupling_H} thus holds in the case $I^L(\eta) = \infty$.
 
    It remains to prove \eqref{eq_coupling_H} assuming $I^L(\eta) < \infty$. For this, using \eqref{eq:feb8.1} it is enough to verify that the cross terms arising from the integrals on the right in \eqref{eq_coupling_H} cancel, that is, we want to show that
      \begin{equation}\label{eq:cross_term}
      \int_{\m H} \brac{\nabla \s_f(z), \nabla (\varphi\circ f) (z)}dz^2 + \int_{\m H^*} \brac{\nabla \s_g(z), \nabla (\varphi \circ g) (z)}dz^2 = 0.
      \end{equation} 
 Let us first assume that $\eta$ is smooth and $\varphi \in C_c^{\infty} (\m C)$. 
    By Stokes' formula, the first term on the left-hand side of \eqref{eq:cross_term} is equal to 
\begin{align*}
 \int_{\m R} \partial_n \s_f(x) \varphi (f (x)) d x & = \int_{\m R} k(f(x)) \abs{f'(x)} \varphi (f(x)) d x
 = \int_{\partial H} k(z) \varphi(z) \abs{dz}
\end{align*}    
where $k(z)$ is the geodesic curvature of $ \eta = \partial H$ at $z$ using the identity
 $\partial_{n} \s_f (x) = |f'(x)|k (f(x)). $
The geodesic curvature at the same point $z \in \eta$ considered as a prime end of $\partial H^*$ equals $-k(z)$. Therefore \eqref{eq:cross_term} follows in the smooth case.

   For the general case, notice that from a change of variable \eqref{eq:cross_term} can be rewritten as
   \begin{equation}\label{eq:cross_term_2}
      \int_{H} \brac{\nabla \s_{f^{-1}}(z), \nabla \varphi(z)}dz^2 + \int_{H^*} \brac{\nabla \s_{g^{-1}}(z), \nabla \varphi(z)}dz^2 = 0.
      \end{equation}

Let $\eta^y = f(\m R +iy)$ which in particular is a smooth curve, let $f_y (\cdot) =  f ( \cdot + iy)$ which is the conformal map from $\m H$ to the component $H_y$ of $\m C \smallsetminus \eta^y$ contained in $H$, and let $g_y$ be a conformal map from $\m H^*$ to the other component $H_y^*$ fixing $\infty$.
It follows from \eqref{eq:cross_term_2} that
\begin{equation} \label{eq:H_y_vanishing}
      \int_{H_y} \brac{\nabla \s_{f_y^{-1}}(z), \nabla \varphi(z)}dz^2 + \int_{H_y^*} \brac{\nabla \s_{g_y^{-1}}(z), \nabla \varphi(z)}dz^2 = 0.
      \end{equation} 
 Let $K$ be a compact set in $\m C \smallsetminus \eta$. 
     Then
     by the Carath\'eodory kernel theorem, $(f_y^{-1})''/(f_y^{-1})'$ converges uniformly to $(f^{-1})''/(f^{-1})'$ in $K$ as $y \to 0$, for $y$ such that $\eta^y \cap K = \emptyset$, and similarly for $g$. Therefore
\begin{align} \label{eq:approx_compact}
\begin{split}
     & \lim_{y \to 0+}    \left( \int_{H_y \cap K} \brac{\nabla \s_{f_y^{-1}}(z), \nabla \varphi(z)}dz^2 + \int_{H_y^*\cap K} \brac{\nabla \s_{g_y^{-1}}(z), \nabla \varphi(z)}dz^2 \right)\\
= &  \int_{H \cap K} \brac{\nabla \s_{f^{-1}}(z), \nabla \varphi(z)}dz^2 + \int_{H^*\cap K} \brac{\nabla \s_{g^{-1}}(z), \nabla \varphi(z)}dz^2.
\end{split}
\end{align}
        On the other hand, Corollary~\ref{cor_approx_H} implies that 
     $$\limsup_{y \to 0+} \left(  \int_{H_y} \abs{\nabla \s_{f_y^{-1}}(z)}^2 dz^2 + \int_{H_y^*} \abs{\nabla \s_{g_y^{-1}}(z)}^2 dz^2  \right)= \limsup_{y \to 0+} \pi I^L(\eta^y) =: C < \infty.  $$
     Hence, by Cauchy-Schwarz,
     \begin{align*}
             &  \limsup_{y \to 0+} \abs{ \int_{H_y \cap K^c} \brac{\nabla \s_{f_y^{-1}}(z), \nabla \varphi(z)}dz^2 + \int_{H_y^*\cap K^c} \brac{\nabla \s_{g_y^{-1}}(z), \nabla \varphi(z)}dz^2 }\\
             \le & \,  C\left(\int_{K^c} \abs{\nabla \varphi (z)}^2 dz^2\right)^{1/2} \to 0,
     \end{align*}
     as $K$ exhausts $\m C \smallsetminus \eta$ and we similarly have,
     \begin{align*}
     \abs{ \int_{H \cap K^c} \brac{\nabla \s_{f^{-1}}(z), \nabla \varphi(z)}dz^2 + \int_{H^*\cap K^c} \brac{\nabla \s_{g^{-1}}(z), \nabla \varphi(z)}dz^2 } \to 0
     \end{align*}
     Together with \eqref{eq:H_y_vanishing} and \eqref{eq:approx_compact} we obtain \eqref{eq:cross_term_2} for all finite energy curves $\eta$ and smooth $\varphi$.
     
 Finally, by approximating $\varphi \in \mc E (\m C)$ by functions in $C^{\infty}_c (\m C)$ in the Dirichlet semi-norm (Lemma~\ref{lem:density}), we have the equality for all finite energy $\eta$ and all $\varphi \in \mc E (\m C)$.
    \end{proof}

\begin{thm}[Isometric conformal welding]\label{thm:tuple_3_2}
Suppose $u \in \mc E(\mathbb{H})$ and $v \in \mc E(\mathbb{H}^*)$ are given with $u, v \in H^{1/2}(\mathbb{R})$ also denoting the corresponding traces on $\mathbb{R}$. Let $h$ be the isometric welding homeomorphism of $\m R$ constructed from the measures $e^{u} dx$ and $e^{v} dx$ as in \eqref{def:homeo1}. There exists a unique normalized solution $(\eta, f, g)$ to the conformal welding problem for $h$.
Moreover, $\eta$ has finite energy and there exists $\varphi \in \mc E(\mathbb{C})$ such that \eqref{eq:def_u_v} and \eqref{eq_coupling_H} hold.
 \end{thm}

\begin{proof}

 Proposition~\ref{prop:isometricwelding} shows that the welding homeomorphism $h$ is quasisymmetric and satisfies $\log h'  \in H^{1/2}(\mathbb{R})$. Therefore the solution to the welding problem is unique, and it follows from Theorem~\ref{thm:h12characterization} that $\eta$ has finite energy.

   Using \eqref{eq_infinite_loewner}, we have $I^L(\eta) = \mc D_{\m H}(\log |f'|) + \mc D_{\m H^*}(\log |g'|)<\infty$. 
 To satisfy \eqref{eq:def_u_v}, we define 
   $$\varphi (z) =
   \begin{cases}
      \left(u - \log \abs{f'}\right) \circ f^{-1}(z), & \text{ if } z \in  H\\
     \left(v - \log  \abs{g'} \right) \circ g^{-1}(z), & \text{ if } z \in  H^*.
   \end{cases} $$
Since the Dirichlet energy is invariant by precomposing with a conformal map, we immediately have $\mc D_{\m C \smallsetminus \eta} (\varphi)< \infty$ since $u$, $v$, $\log \abs{f'}$ and $\log \abs{g'}$ all have finite Dirichlet energy.
Hence we only need to check that  
$\vp \in \mc E (\m C)$. Indeed, if  $\vp \in \mc E (\m C)$, then
$\mc D_{\m C \smallsetminus \eta} (\phi) = \mc D_{\m C} (\phi)$
since $\eta$ has Lebesgue measure $0$. Therefore $\mc D_{\m C} (\phi) < \infty$, so Theorem~\ref{thm_coupling_identity_H} applies and  \eqref{eq_coupling_H} follows.

 We want to use the gluing Lemma~\ref{lem:sobolev_gluing} and so we need to check that the traces of $\vp$ from both sides of the curve match. By Lemma~\ref{lem:trace_commutation} (precomposing by a M\"obius transformation), the conformal map $f^{-1}$ commutes with the trace operator, and for arclength a.e.  $z\in \eta$,
 $$\varphi_+ (z) : = \mc R_{H \to \eta} [\varphi] (z) =  \left(\mc R_{\m H \to \m R} \left[u - \log \abs{f'}\right]\right) \circ f^{-1}(z) = \left(u - \log \abs{f'}\right) \circ f^{-1}(z). $$
 
 In the last expression above, $f'(x)$ stands for the non-tangential limit of $f'$ at $x \in \m R$ and the last equality follows from linearity and Lemma~\ref{lem:traces_are_identical}.
 
 In fact, $f'(x)$ coincides with the tangential derivative of $f$ a.e. and the arclength of $f(I) \subset \eta$ is given by $\int_I |f'(x)|d x$ for  any interval $I$ of $\m R$, see \cite[ Theorem~6.8]{pommerenke}.
 Similarly,
 $$\varphi_- (z) : = \mc R_{H^* \to \eta} [\varphi] (z) =  \left(v - \log \abs{g'}\right) \circ g^{-1}(z). $$
Since $h$ is the welding homeomorphism for $\eta$, we have by construction
$
\int_{I} e^{u} dx = \int_{h(I)} e^{v} dx.
$
On the other hand, since $\eta$ is locally rectifiable, 
we have 
\[
\int_{I} e^{u (x)} dx = \int_{f(I)} e^{u \circ f^{-1} (z)} |(f^{-1})'(z)||dz| = \int_{f(I)} e^{\varphi_+ (z)} |dz|, \]
and similarly
$  \int_{h(I)} e^{v} dx = \int_{g(h(I))} e^{\phi_-} |dz|.
$
It follows that for every choice of $a,b$,
\[
\int_{\eta[a,b]} e^{\phi_+} |dz| = \int_{\eta[a,b]} e^{\phi_-} |dz|,
\]
which implies $\phi_+ = \phi_-$ a.e. on $\eta$. 
From Lemma~\ref{lem:sobolev_gluing}, we conclude that $\varphi$ extends to a function in $\mc E (\m C)$
and this completes the proof.
\end{proof}

\begin{rem}
 The proof of Theorem~\ref{thm:tuple_3_2} shows that the measure $e^{\phi} |dz|$ on $\eta$ equals the pushforward of both  $e^u dx$ by $f$ and $e^v dx$ by $g$ on $\m R$. 
   This is the analog of the equality of the chaos measure with respect to Minkowski content on the SLE path and the pushforward of the boundary LQG measures in the quantum zipper, see \cite{Benoist}. 
\end{rem}

\begin{cor} \label{cor:3_2_alternative}
Suppose $u \in \mc E(\mathbb{H})$ and  $v \in \mc E(\mathbb{H}^*)$ are given.
    Then there exists a unique tuple $(\varphi, \eta,  f, g)$ such that:
    \begin{enumerate}[topsep=-.5pt,itemsep=-3pt]
    \item $\eta$ is a Jordan curve passing through $0$, $1$ and $\infty$;
    \item  $f: \m H \to H$ is the conformal map fixing $0,1$ and $\infty$ and $g: \m H^* \to H^*$ is a conformal map fixing $0, \infty$;
   \item \label{item_varphi_matching} $\varphi \in \mc E(\mathbb{C})$ and  \eqref{eq:def_u_v} holds.
    \end{enumerate}
Moreover, $\eta$ is obtained from the isometric conformal welding of $\m H$ and $\m H^*$ according to the boundary lengths $e^{u}dx$ and $e^{v}dx$. 
 \end{cor}

\begin{proof}
We only need to show that $\eta$ is necessarily obtained from the isometric welding of $e^u dx$ and $e^v dx$, since Theorem~\ref{thm:tuple_3_2} then implies the existence and uniqueness of the tuple as $f$ is normalized to fix $0,1,\infty$ and $\eta$ is a quasicircle and therefore conformally removable.

Let $( \varphi, \eta, f, g)$ be any tuple satisfying the above conditions. Then Theorem~\ref{thm_coupling_identity_H} implies that 
$$I^L(\eta) =   \mc D_{\m H} (u) + \mc D_{\m H^*} (v) - \mc D_{\m C} (\varphi)<\infty.$$
It follows that $\eta$ is chord-arc and since $\vp \in \mc E(\m C)$, its trace  $\phi|_\eta \in H^{1/2}(\eta)$ as in Appendix~\ref{sect:trace}. 
  As in the proof of Theorem~\ref{thm:tuple_3_2}, the length of a portion of $\eta$ using the corresponding metric can be computed as
  \begin{align*}
      \int_{\eta[a,b]} e^{ \phi|_\eta(z)} \abs{dz} &= \int_{f^{-1}(\eta[a,b])} e^{ \phi|_\eta  \circ f(x)} \abs{f'(x)} dx \\
      &= \int_{f^{-1}(\eta[a,b])} e^{u(x)} dx=\int_{g^{-1}(\eta[a,b])} e^{v(x)}dx.
  \end{align*}
      Therefore $f$ and $g$ are given by the isometric welding of the measures $e^{u} dx $ and $e^{v} dx$. 
\end{proof}

\begin{rem}Note that by Theorem~\ref{thm:tuple_3_2}, in that same setup, \eqref{eq_coupling_H}   shows that 
\begin{equation} \label{eq:energy_comparison_u_v}
    I^L(\eta) = \mc D_{\m H} (u) + \mc D_{\m H^*} (v) - \mc D_{\m C} (\varphi)  \le
     \mc D_{\m H} (u) + \mc D_{\m H^*} (v).
\end{equation}
\end{rem}

We next show that finite energy curves are closed under arclength isometric welding (see Figure~\ref{fig:arc_welding}) with the energy of the welding curves bounded above by the sum of the energies of the initial pair of curves. 
We can view this inequality as quantifying the  dissipation of energy into the global \field~$\varphi$.

\begin{figure}[ht]
 \centering
 \includegraphics[width=0.6\textwidth]{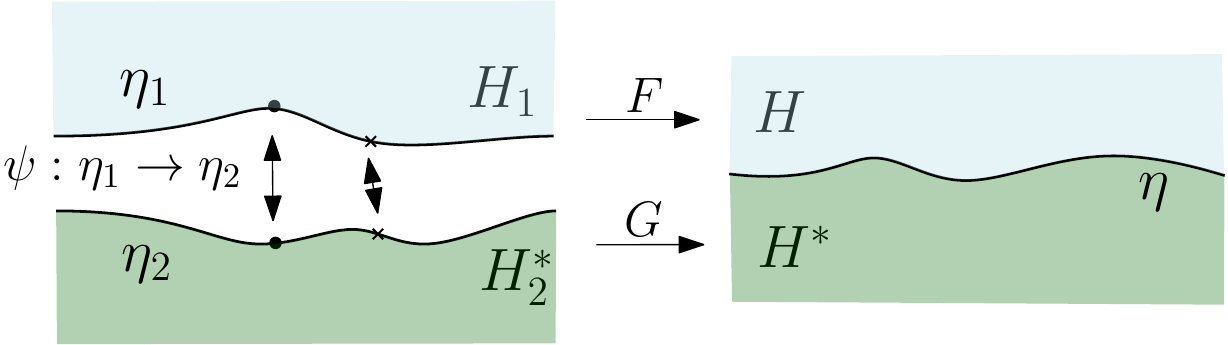}
 \caption{\label{fig:arc_welding} Arclength isometric welding of $H_1$ and $H_2^*$. The isometry $\psi = G^{-1} \circ F|_{\eta_1}$ identifies arcs of the same length and $F, G$ are conformal maps. The welding curve $\eta$ has energy bounded by the sum of the energies of $\eta_1$ and $\eta_2$.} 
 \end{figure}

More precisely, let $\eta_1$ and $\eta_2$ be two Jordan curves through $\infty$ with finite energy. Let $H_1, H_1^*$ be the connected components of $\m C \smallsetminus \eta_1$ and $H_2, H_2^*$ the connected components of $\m C \smallsetminus \eta_2$.
\begin{cor} \label{cor:arc-length}
   Let $\eta$ (resp. $\tilde \eta$) be the arclength isometric welding curve of the domains $H_1$ and $H_2^*$ (resp.  $H_2$ and $H_1^*$). 
  Then $\eta$ and $\tilde \eta$ have finite energy. 
  Moreover, 
  $$I^L(\eta)  + I^L(\tilde \eta) \le I^L(\eta_1) + I^L(\eta_2).$$
\end{cor}

\begin{proof} For $i=1,2$, let $f_i$ be a conformal equivalence $\m H \to H_i$, and $g_i:  \m H^* \to H_i^*$ both fixing $\infty$.  
By \eqref{eq_infinite_loewner},
$$I^L(\eta_i) = \mc D_{\m H}\left(\log |f_i'| \right) + \mc D_{\m H^*}\left(\log |g_i'|\right).$$

Set $u_i : = \log |f'_i|, v_i := \log |g'_i|$. Then $\eta$ is the welding curve obtained from Theorem~\ref{thm:tuple_3_2} with $u = u_1$, $v= v_2$ and $\tilde \eta$ is the welding curve for $u = u_2$, $v= v_1$.
Then \eqref{eq:energy_comparison_u_v} implies
\begin{align*}
    I^L(\eta) + I^L(\tilde \eta)  \le \mc D_{\m H}\left(u_1\right) + \mc D_{\m H^*}\left(v_2\right) + 
\mc D_{\m H}\left(u_2\right) + \mc D_{\m H^*}\left(v_1\right)  =  I^L(\eta_1) + I^L(\eta_2)
\end{align*}
as claimed.
\end{proof}

\begin{rem}
  In particular, we have the energy sub-additivity: $$I^L(\eta) \le  I^L(\eta_1) + I^L(\eta_2)$$
  and equality holds only when $\mc D_{\m H}\left(u_2\right) = \mc D_{\m H^*}\left(v_1\right) = 0$ which implies that both $\eta_1$ and $\eta_2$ are lines.
\end{rem}

Given Corollary~\ref{cor:arc-length}, it is natural to ask for when $\vp$ is constant so that \eqref{eq:energy_comparison_u_v} becomes an equality. 
The following proposition aims to provide the geometrical intuition that this happens if and only if  $e^{2u} dz^2$ and $e^{2v} dz^2$, considered as Riemannian metrics, have matching geodesic curvatures at points identified by the welding homeomorphism, and are flat in the bulk ($\D u = \D v = 0$). We restrict ourselves to the smooth case in order to simplify the discussion and to have all the quantities well-defined.

In the statement, $\partial_{n}$ and $\partial_{n^*}$ denote the outer normal derivative on $\m R = \partial \m H = \partial \m H^*$ with respect to $\m H$ and $\m H^*$. 

\begin{prop}[Curvature matching] Let $u \in C^{\infty} (\ad {\m H}) \cap \mc E (\m H)$ and $v \in C^{\infty} (\ad {\m H}^*) \cap \mc E (\m H^*)$.
   The \field~$\varphi$ obtained in Theorem~\ref{thm:tuple_3_2} satisfies
   $\mc D_{\mathbb{C}} (\varphi) = 0$ if and only if $u$ and $v$ are harmonic, and $\partial_{n} u +  (\partial_{n^*} v)\circ h \cdot h'  \equiv 0$ on $\m R$.
\end{prop}

\begin{proof}
The ``only if'' part:
since $\varphi$ is constant, we have that
$u = \log \abs{f'} + \varphi \circ f$ is harmonic and so is $v = \log |g'| +\varphi \circ g$.
Let $k (y)$ (resp. $k^* (y)$) be the geodesic curvature at $y \in \partial H$  (resp. $ y \in \partial H^*$) under the metric $e^{2\phi} dy^2$ which also equals the geodesic curvature at $z  = f^{-1}(y) \in \partial \m H$ under the metric $e^{2u} dz^2$.
We have the following identity for all $z \in \m R$,
 $$k (f(z)) =  e^{-u (z)} \left( k_0(z)  + \partial_{n} u (z) \right) =  e^{-u (z)} \partial_{n} u (z), $$
 where $k_0 \equiv 0$ is the geodesic curvature of $\m R$ as boundary of $\m H$ under the Euclidean metric. 
 Hence for all intervals $I \subset \m R$,
$$\int_I \partial_{n} u (z) dz =  \int_I e^{u(z)} k (f(z)) dz = \int_I |f'(z)| k (f(z)) e^{\phi (f(z))} dz =  \int_{f(I)} k (y) e^{\phi (y)}  |dy|.  $$
Similarly for $ \partial_{n^*} v$, since $\phi$ is constant, we have
$$\int_I \partial_{n} u (z) dz = \int_{f(I)} k (y)  e^{\phi (y)}  |dy| = -\int_{g \circ h(I)} k^* (y) e^{\phi (y)}  |dy|  = - \int_I (\partial_{n^*} v) \circ h(z) h'(z)dz.$$
 It follows that $\partial_{n} u +  (\partial_{n^*} v)\circ h \cdot h'  \equiv 0$, as claimed.

The ``if'' part:
we check that $\varphi$ is harmonic everywhere in $\m C$.
Let $\rho \in C_c^{\infty} (\m C)$ be a test function,
\begin{align*}
  \brac{\nabla \varphi, \nabla \rho}_{\m C} = & \brac{\nabla (\varphi \circ f), \nabla (\rho \circ f)}_{\m H} + \brac{\nabla (\varphi \circ g), \nabla (\rho \circ g)}_{\m H^*}\\
  = & \brac{\nabla u, \nabla (\rho \circ f)}_{\m H} + \brac{\nabla v, \nabla (\rho \circ g)}_{\m H^*}\\
  = & \brac{\partial_{n} u, \rho \circ f}_{\m R} + \brac{\partial_{n^*} v, \rho \circ g}_{\m R}.
\end{align*}

The second equality follows from 
$$\brac{\nabla \log \abs{f'}, \nabla (\rho\circ f)}_{\m H} + \brac{\nabla \log \abs{g'}, \nabla (\rho\circ g)}_{\m H^*} = 0$$
as in \eqref{eq:cross_term}
   and the third equality above follows from the assumption that $u$ and $v$ are harmonic.
   We have also that 
   $$\brac{\partial_{n^*} v, \rho \circ g}_{\m R} = \brac{\partial_{n^*} v, \rho \circ f \circ h^{-1}}_{\m R} = \brac{(\partial_{n^*} v) \circ h\cdot h', \rho \circ f}_{\m R}.  $$
   Since we assumed $\partial_{n} u +  (\partial_{n^*} v)\circ h \cdot h'  \equiv 0$, we have
   $\brac{\nabla \varphi, \nabla \rho} = 0$
  for all $\rho$. 
  It follows that $\varphi$ is harmonic in $\m C$. Since the Dirichlet energy of $\varphi$ is finite, $\varphi$ is constant.   
\end{proof}

\subsection{Welding identity: disk version} \label{sec:D_welding}
We will now discuss the welding identity in the case when $\eta$ is a bounded finite energy curve. Denote by $\O$ and $\O^*$ the bounded and unbounded connected components of $\m C \smallsetminus \eta$ and let $\varphi \in \mathcal{E}(\mathbb{C})$.
 As in the half-plane case, we associate to the pair $(\varphi, \eta)$ two \field{}s, this time defined on $\mathbb{D}$ and $\mathbb{D}^*$:
    $$ u = \varphi \circ f + \log \abs{f'}  , \quad v =  \varphi \circ g + \log \abs{g'},$$
    where $f : \mathbb{D} \to \O$ and $g: \mathbb{D}^* \to \O^*$ represent some choice of Riemann maps, such that $g (\infty) = \infty$.
It turns out that the correct action functional for the analog of Theorem~\ref{thm_coupling_identity_H} in the disk setting has an extra curvature term. (Or rather, that term is identically $0$ in the half-plane case.) More precisely, if $u: \overline{\Omega} \to \mathbb{R}$, we define
    $${\mc S}_\Omega (u) : = \mc D_\Omega (u) + \frac{2}{\pi} \int_{\partial \Omega } k_{\Omega }(z) u(z) dz$$
    where $k_\Omega$ is the geodesic curvature (using the Euclidean metric) of $\partial \Omega$. Note that we have $ k_{\m D} = -k_{\m D^*}  \equiv 1$, so \eqref{eq_disk_energy} can be written as
    \begin{equation} \label{eq:D_energy_in_S}
        I^L(\eta) = \mc S_{\m D} (\log |f'|) + \mc S_{\m D^*} (\log |g'|).
    \end{equation}

 \begin{thm}\label{thm_coupling_identity_D}Suppose $\varphi \in \mathcal{E}(\mathbb{C})$ and that $\eta$ is a bounded finite energy curve.
 Then we have the identity:
   \begin{equation}
   \label{eq_coupling_D}
   \mc D_{\mathbb{C}}(\varphi) + I^L(\eta)=    {\mc S}_{\mathbb{D}}(u) +  {\mc S}_{\mathbb{D}^*}(v).
   \end{equation}
\end{thm}

 \begin{proof}
   It suffices to prove that the cross-terms in the Dirichlet inner product satisfy
    \begin{align*}
  \frac{2}{\pi}  \int_{\m D} & \brac{\nabla \log \abs{f'(z)}, \nabla \varphi(f(z))}dz^2 + \frac{2}{\pi} \int_{\m D^*} \brac{\nabla \log \abs{g'(z)}, \nabla \varphi(g(z))}dz^2 \\
    & =  -\frac{2}{\pi} \int_{\partial \m D} \varphi (f(z)) dz + \frac{2}{\pi} \int_{\partial \m D^*} \varphi (g(z)) dz.
    \end{align*}
   Assume first that $\eta$ is smooth and $\varphi \in C^{\infty}_c (\m C)$. 
    Using Stokes' formula, the first term on the left-hand side is equal to 
\begin{align*}
 & \frac{2}{\pi}\int_{\partial \m  D} \partial_n \log \abs{f'(z)} \varphi (f(z)) d z\\
 = & \frac{2}{\pi} \int_{\partial \m D} k_\O(f(z))\abs{f'(z)} \varphi (f(z)) d z - \frac{2}{\pi} \int_{\partial \m D} k_{\m D}(z) \varphi (f(z)) d z \\
 = &\frac{2}{\pi} \int_{\partial \O} k_\O(y) \varphi (y) d y - \frac{2}{\pi} \int_{\partial \m D} \varphi (f(z)) d z ,
\end{align*}
since 
$$\partial_n \log \abs{f' (z)} = k_\O (f(z))\abs{f' (z)} - k_{\m D} (z). $$
As $k_\O (y) = - k_{\O^*} (y)$ for all $y \in \eta$, we have 
$$\int_{\partial \O} k_\O(y) \varphi (y) d y + \int_{\partial \O^*} k_{\O^*}(y) \varphi (y) d y = 0$$ 
which concludes the proof in the smooth case.
The approximation of a general finite energy $\eta$ by equipotentials and $\phi \in \mc E(\m C)$ by $C^{\infty}_c (\m C)$ is the same as in the proof of Theorem~\ref{thm_coupling_identity_H}.
    \end{proof}

    For $u, v \in H^{1/2} (\m T)$, $e^u |dz|$ and $e^v |dz|$ define finite measures on $\m T$.
We normalize them to have total mass $1$ by subtracting (from $u$) $z_u : = \log \int_{\m T} e^u |dz|$ and (from $v$) $z_v : = \log \int_{\m T} e^v |dz|$.

In a similar manner as in \eqref{def:homeo1}, we define a homeomorphism $h$ of $\m T$ which isometrically identifies $e^{u - z_u} |dz|$ and $e^{v-z_v} |dz|$ and we may assume it fixes $1$. 
A normalized solution of the conformal welding problem for $h$ is a triple $(\eta, f, g)$, where $\eta$ is a Jordan curve in $\m C$ with associated  $\O, \O^*$ and conformal maps $f : \m D \to \O$ fixing $0, 1$, and $g: \m D^* \to \O^*$ fixing $\infty, 1$ such that $h = g^{-1} \circ f|_{\m T}$.
We also have $\log |h'| \in H^{1/2}(\m T)$ by the proof of Proposition~\ref{prop:isometricwelding}. Theorem~\ref{thm_TT_equiv_T01} implies that $I^L(\eta) < \infty$.
    The proof of Theorem~\ref{thm:tuple_3_2} then gives:
    
    \begin{thm}[Isometric welding of disks]\label{thm:tuple_D}
Suppose $u \in \mc E(\mathbb{D})$ and $v \in \mc E(\mathbb{D}^*)$ are given with $u, v \in H^{1/2}(\m T)$ also denoting the corresponding traces on $\m T$. 
Let $h$ be the isometric welding homeomorphism of $\m T$ constructed as above, and $(\eta, f, g)$ the normalized solution.
Then there exists a unique $\varphi \in \mc E(\mathbb{C})$ such that
     $$ u - z_u = \varphi \circ f + \log \abs{f'}  , \quad v -  z_v =  \varphi \circ g + \log \abs{g'},$$
 and 
 $$\mc D_{\mathbb{C}}(\varphi) + I^L(\eta)=    {\mc S}_{\mathbb{D}}(u -z_u) +  {\mc S}_{\mathbb{D}^*}(v-z_v).$$
 \end{thm}
 
The analog of Corollary~\ref{cor:arc-length} also holds:
\begin{cor}\label{cor:arc-length-D}
  Let $\eta_1$ and $\eta_2$ be two finite energy curves with the same arclength, let $\O_i$ and $\O_i^*$ be associated to $\eta_i$,
  and let $\eta$ (resp. $\tilde \eta$) be the isometric welding curve of $\O_1$ and $\O_2^*$ (resp. $\O_2$ and $\O_1^*$). Then 
  $$I^L(\eta) + I^L(\tilde \eta) \le I^L(\eta_1) + I^L(\eta_2).$$
\end{cor}

\begin{proof}
Without loss of generality, we assume that the arclength of $\eta_1$ and $\eta_2$ are $1$.
   We put for $ i  = 1,2$, $u_i : = \log |f'_i|, v_i := \log |g'_i|$. Then $\eta$ is the welding curve given by Theorem~\ref{thm:tuple_D} with $u = u_1$ and $v = v_2$.
Similarly $\tilde \eta$ corresponds to $u=u_2, v=v_1$. We obtain
\begin{align*}
    I^L(\eta) + I^L(\tilde \eta)  \le \mc S(u_1) + \mc S(v_2) + \mc S(u_2) + \mc S(v_1)  =  I^L(\eta_1) + I^L(\eta_2)
\end{align*}
 as claimed since $z_{u_i} = z_{v_i} = 0$.
\end{proof}

In contrast with the cutting and welding operations, our flow-line identity is specific to the half-plane setting for the Loewner energy: as we will see, all the flow-lines are ``bi-infinite'' and go through $\infty$.

\subsection{Flow-line identity}\label{sec:flow_line}

Let $\eta=\eta(s)$ be an asymptotically smooth Jordan curve through $\infty$, parametrized by arclength.
 Let $H, H^*$ be the two connected components of $\m C \smallsetminus \eta$ and let $f : \m H \to H$ and $g: \m H^* \to H^*$ be conformal maps fixing $\infty$. We choose $\arg f'$ to be a (fixed) continuous branch of $\Im \log f'$ in $\m H$. 
  By Theorem~II.4.2 of \cite{GM}, for a.e. $\zeta = \eta(s)$ such that $\eta'(s)$ exists, $\arg f' (z)$ has a limit as $z$ approaches $f^{-1} (\zeta)$ in $\m H$ and we denote this limit by $\tau(\zeta)$. Moreover,
\begin{equation}\label{eq:tau_def}
 \eta'(s) = \lim_{t\to s} \frac{\eta(t) - \zeta}{t-s} = \pm \lim_{t \to s \pm} \frac{\eta(t) -\zeta}{|\eta(t) - \zeta| } = e^{i \tau (\zeta)}.
 \end{equation}
 The second equality uses that $\eta$ is asymptotically smooth. Having chosen a branch of $\arg f'$, we choose one for $\arg g'$ defined on $\mathbb{H}^*$ so that the boundary values of $\arg g' \circ g^{-1}$ agree with $\tau$ a.e. Finally, let $\mc P[\tau]$ be the Poisson extension of $\tau$ to $\m C \smallsetminus \eta$.

 \begin{lemma}\label{lem:arg_equals_tau}
 Suppose $\eta$ is a finite energy curve through $\infty$. Then we have
 \[
\arg f' (z) = \mc P[\tau] \circ f(z), \quad \forall z \in  {\m H}.
 \]
 \end{lemma}

 \begin{proof}
Since $\mc{D}_{\HH}(\arg f') = \mc D_{\m H} (\log |f'|)< \infty$, it follows that (a.e.) $\tau \circ f = \arg f' \in H^{1/2}(\m{R})$.
Since the trace operator $\mc E_{\text{harm}} (\m H) \to H^{1/2} (\m R)$ is one-to-one, using Lemma~\ref{lem:trace_commutation} we see that $ \mc{P}[\tau]  \circ f  = \mc{P}[\tau \circ f] = \arg f'$.
\end{proof}

\begin{thm} \label{thm_flow_line}
If $\eta$ is a finite energy curve through $\infty$, we have the identity 
\begin{equation}\label{eq_flow_identity}
I^L(\eta) = \mc D_{\m C} (\mc P[\tau]).
\end{equation}
Conversely, if $\varphi \in \mc E(\m C)$ is continuous and $\lim_{z \to  \infty} \varphi(z)$ exists, then for all $z_0 \in \m C$, any solution to the differential equation
$$ \dot{\eta}(t) = \exp\left(i\varphi (\eta(t))\right),\quad  t\in (-\infty, \infty) \quad \text{and} \quad \eta(0) = z_0 $$
is a $C^1$ Jordan curve through $\infty$ with finite energy, and
$$ \mc D_{\m C} (\varphi) = I^L(\eta) + \mc D_{\m C} (\varphi_0),$$
where $ \vp_0=\vp- \mc P [\varphi|_{\eta}]$.
\end{thm}

\begin{proof}
Let $f$ and $g$ be conformal maps from $\m H$ and $\m H^*$, respectively, associated to $\eta$ as above.
From Lemma~\ref{lem:arg_equals_tau}, we have that
$\arg f' (z) = \mc P[\tau] \circ f(z)$ for all $z \in  {\m H}$ and similarly for $g$. 
The identity \eqref{eq_flow_identity}, with $\mc D_{\m C}$ replaced by $\mc D_{\m C \smallsetminus \eta}$, then follows from \eqref{eq_infinite_loewner} using that
$\mc D_{\m H} (\log \abs{f'}) = \mc D_{\m H} (\arg f')$
and 
$\mc D_{\m H} (\mc P[\tau] \circ f) = \mc D_{f(\m H)} (\mc P[\tau])$, together with the analogous formulas for $g$. Since the traces of the harmonic extensions of $\tau$ to $H$ and $H^*$ agree almost everywhere, the gluing Lemma~\ref{lem:sobolev_gluing} shows that we obtain a function in $\mc{E}(\m C)$ and \eqref{eq_flow_identity} holds.

For the converse statement, notice that $\exp(i \varphi)$ defines a continuous unit vector field on $\m C$. By the Cauchy-Peano existence theorem there exists a solution $\eta$ (which may not be unique) for all $t \in \m R$ and $t \to \eta(t)$ is an arclength parametrized $C^1$ curve.

We claim that the solution $\eta$ is a Jordan curve through $\infty$. We first prove that $\eta$ contains no closed loop in $\m C$.
In order to derive a contradiction, assume that $\eta(0) = \eta(1)$ and $ [0,1) \to \eta[0,1)$ is injective. Since $\eta[0,1]$ is a bounded Jordan curve, it encloses a bounded simply connected domain $\O$ and we assume that $\eta$ winds counterclockwise around $\O$ (consider $\varphi + \pi$ otherwise). Let $\psi:\m D \to \O$ be a conformal map.  Since the vector field is continuous,  $\partial \O$ has a continuous tangent, so $\arg \psi'$  extends continuously to $\ad {\m D}$ and
$$\exp (i \varphi \circ \psi(z)) = i z \exp (i \arg \psi' (z)), \quad \forall z \in \m T.$$ However, $\log z$ does not have a continuous branch on $\m T$ but since $i (\varphi \circ \psi (z) - \arg \psi'(z) - \pi/2)$ would provide one, we have a contradiction.

We next show that $\eta$ is transient as $t \to \pm 
\infty$.
     Assume this is not the case. Then since $\eta$ is a flow-line of a continuous unit vector field, there exists $z \in \m C$ such that for all $r>0$, $\eta$ visits the closed ball $B(z,r)$ at least twice. Since $\varphi$ is continuous, there is $r>0$ such that $w \in B(z,100r)$ implies $|\varphi (w) - \varphi(z)|< 1/10$. After $\eta$ visits  $B(z,r)$ for the first time (of the two times), $\eta$ leaves $B(z,100r)$ from the sub-arc of argument $ [\varphi(z)- 1/9, \varphi(z)+ 1/9]$ of $\partial B(z,100r)$ and re-enters  $B(z,100r)$ from the arc of argument in $ [\varphi(z)+ \pi/2 - 1/10, \varphi(z) + 3\pi/2 + 1/10]$.
     We call the exit time $s$ and the re-enter time $t$. 
Now we  modify $\varphi$ inside $B(z, 100r)$ such that $\varphi$ remains continuous and the unit vector field $e^{i\varphi}$ generates a flow $\tilde \eta$ starting from $\tilde \eta (0) = \eta(t)$ hits $\eta (s)$ at some time $\d >0$, with $\tilde \eta [0,\d] \subset B(z,100r)$. But the existence of the loop $\eta [s,t] \cup \tilde \eta [0,\d]$ then contradicts the fact, proved as in the previous paragraph, that the flow of the modified continuous vector field contains no closed loop in $\m C$.

Therefore $\eta$ is an infinite $C^1$ simple curve and since $\lim_{z \to \infty} \varphi(z)$ exists, $\eta$ is in fact the M\"obius image of a bounded $C^1$ Jordan curve. 
 It follows that $\arg f'$ is bounded and harmonic, so $\mc P[\arg f'|_{\m R}] = \arg f'$. Using \eqref{eq:tau_def} again, we obtain $(\arg f')\circ f^{-1}|_{\eta} = \varphi|_{\eta}$ and 
 \begin{align*}
     \mc D_{\m H} (\arg f')& = \mc D_{\m H} (\mc P[\arg f'|_{\m R}]) =  \mc D_{H} (\mc P[\arg f' \circ f^{-1}|_{\eta}]) = \mc D_{H} \big(\mc P[\varphi|_{\eta}]\big)\\
    & =  \mc D_{H} (\varphi) - \mc D_{H} (\varphi_0), 
 \end{align*}
 where the last equality follows from the orthogonal decomposition (for the Dirichlet inner product) as in Lemma~\ref{lem:decomp} (after conformally mapping to a disk). We conclude the proof by performing the same computation with $g$ and then using \eqref{eq_infinite_loewner}.
\end{proof}

The following corollaries are immediate consequences of the flow-line identity. We consider first the family of analytic curves  $\eta^{r} := f(\m R + ir)$, where $r > 0$.
\begin{cor}\label{cor_compare_horo}
Let $\eta$ be finite energy curve through $\infty$.
   For $0 < r < s$, 
   $I^L(\eta^s) \le I^L(\eta^r) \le I^L(\eta)$
   and any equality holds if and only if $\eta$ is a line.
\end{cor}

\begin{proof}
 By Lemma~\ref{lem:arg_equals_tau}, $\mc P [\tau ] \circ f (z) = \arg f'(z)$, so for each $r$, $\eta^r$ 
 has tangent with argument given by  $\mc P [\tau]$ and we write it as
  $\tau^r : = \mc P [\tau]|_{\eta^r}$.
 It follows from Theorem~\ref{thm_flow_line} and the fact that $I^L (\eta^r) <\infty$, 
 $$I^L (\eta^r) = \mc D_{\m C} (\mc P [\tau^r] )= \mc D_{\m C} \Big (\mc P \Big[\mc P [\tau]\Big|_{\eta^r}\Big] \Big) \le \mc D_{\m C} (\mc P[\tau]) = I^L(\eta).$$ 
 The ineqality uses the Dirichlet principle.
   
   In case of equality, $\mc P[\tau]$ is harmonic in the complement of $\eta^r$. Since it is also harmonic in the complement of $\eta$, it follows that $\mc P[\tau] \in \mc E_{\text{harm}} (\m C)$, and consequently constant. 
\end{proof}

    \begin{figure}[ht]
 \centering
 \includegraphics[width=0.5\textwidth]{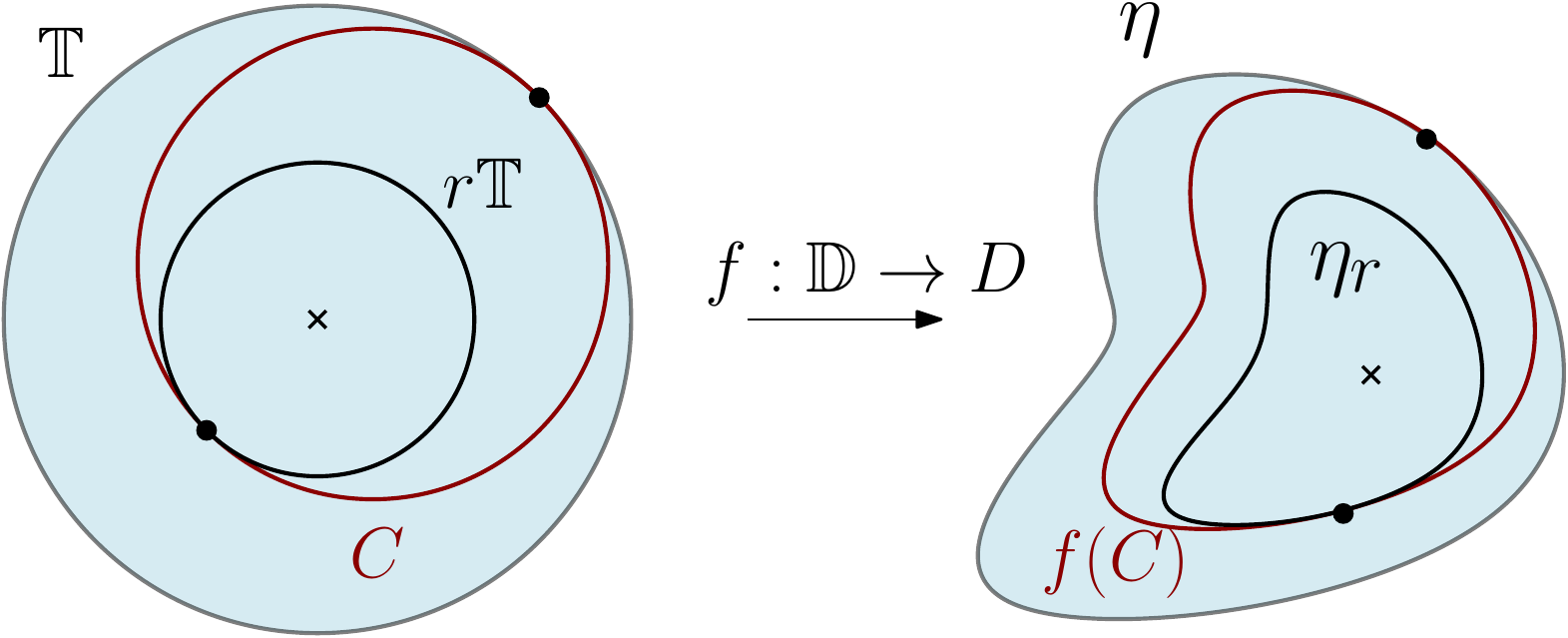}
 \caption{\label{fig_circles} Comparing the Loewner energy of equipotentials.} 
 \end{figure}
 
Corollary~\ref{cor_compare_horo} compares the Loewner energy of the image of a horocycle in the upper half-plane which touches $\infty$. 
If we map $\m H$ to $\m D$ by a M\"obius transformation, the horocycle $\m R + ir$ is mapped to a circle tangent to $\m T$.  
This allows us to compare the energies of equipotentials inside of a bounded domain as follows.

Assume now that  $\eta$ is a bounded finite energy curve. Let $\O$ be the bounded component of $\m C \smallsetminus \eta$, $f$ a conformal map from $\m D \to \O$ and $\eta_r =  f(r \m T)$, for $0 <r<1 $.
Using Corollary~\ref{cor_compare_horo}, we compare  $I^L(\eta)$ and  $I^L(\eta_r)$ to $I^L(f(C))$ where $C$ is a circle tangent to both $\m T$ and $r \m T$ as shown in Figure~\ref{fig_circles} and obtain:

\begin{cor}\label{cor:compare_equipotential} For $0< s < r < 1$,
$
I^L (\eta_s) \le I^L (\eta_r) \le I^L(\eta)
$ and any equality holds only when $\eta$ is a circle.
\end{cor}

\subsection{Complex-valued function identity}

Let us conclude with the following identity which combines both welding and flow-line identities.
  
  Let $\psi$ be a complex-valued function on $\m C$ with finite Dirichlet energy. We assume that $\Im \psi \in C^0(\Chat)$.  
We say that $\eta$ is a flow-line of the vector field $e^{\psi}$ if $\eta$ is a flow-line of $e^{i \Im \psi}$. (The parametrization of $\eta$ will not matter for our purposes.)
Let $f, g$ be conformal maps associated to $\eta$
as in Section~\ref{sec:flow_line}.

\begin{cor}\label{cor:complex_field}
   Let $\eta$ be any flow-line of the complex field $e^{\psi}$. Define $\zeta = \psi \circ f + (\log f')^*$ and $\xi = \psi \circ g + (\log g')^*$.
  Then we have
  $$ \mc D_{\m C}(\psi) = \mc D_{\m H}(\zeta) + \mc D_{\m H^*}(\xi).$$
\end{cor}

\begin{rem}
From Corollary~\ref{cor:complex_field} we can easily recover the flow-line identity, by taking $\Im \psi = \varphi$ and $\Re (\psi) = 0$. 
Similarly, the welding identity follows from taking $\Re \psi = \varphi$ and $\Im \psi = \mc P [\tau]$ where $\tau$ is the winding of the curve $\eta$ as defined in Section~\ref{sec:flow_line}.   
\end{rem}

\begin{proof}
   By Theorem~\ref{thm_flow_line}, $\eta$ has finite Loewner energy and Lemma~\ref{lem:arg_equals_tau} shows that 
   $$\arg f' (z) = \mc P [\Im \psi] \circ f(z), \quad \forall z \in \m H.$$
   Hence we can write, 
   \begin{align*}
       \zeta  &= \left( \Re \psi \circ f + \log |f'|\right) + i  \left( \Im \psi \circ f - \arg f'\right) = u + i  \Im \psi_0 \circ f; \\
       \xi &= v + i \Im \psi_0 \circ g,
   \end{align*}
   where $u : = \Re \psi \circ f + \log |f'| \in \mc E (\m H)$, $v : = \Re \psi \circ g + \log |g'| \in \mc E (\m H^*)$ and $\psi_0 = \psi - \mc P[\psi|_{\eta}]$.
   From the welding identity, we have 
   $$\mc D_{\m C} (\Re \psi) + I^L(\eta) = \mc D_{\m H} (u) + \mc D_{\m H^*} (v).$$
   On the other hand, the flow-line identity gives
   $\mc D_{\m C} (\Im \psi) = I^L(\eta) + \mc D_{\m C} (\Im \psi_0).$
   Hence,
   \begin{align*}
       \mc D_{\m C} (\psi) &= \mc D_{\m C} (\Re \psi) +\mc D_{\m C} (\Im \psi) =  \mc D_{\m C} (\Re \psi) + I^L(\eta) + \mc D_{\m C} (\Im \psi_0) \\
       &=   \mc D_{\m H} (u) + \mc D_{\m H^*} (v) + \mc D_{\m C} (\Im \psi_0)\\
       & =\mc D_{\m H}(\zeta) + \mc D_{\m H^*}(\xi)
   \end{align*}
   as claimed.
\end{proof}
\vspace{10pt}

\appendix

\section{Trace operators on chord-arc curves}\label{sect:trace}
The Sobolev space trace operator is usually defined for domains with Lipschitz boundary. We are interested in domains bounded by finite energy curves (see Section~\ref{sect:Loewner}), which are chord-arc but not necessarily Lipschitz \cite{RW}. This appendix recalls and develops the facts needed for this paper.

It will be convenient to work in the class of chord-arc domains, that is, simply connected domains whose boundary $\eta = \partial \O$ is a chord-arc curve in $\hat{\mathbb{C}}$.
We will follow Jonsson and Wallin \cite{JW1984} to define for $u \in \mc E (\m C)$ a trace on $\eta$ by considering averages over balls and prove the gluing lemma (Lemma \ref{lem:sobolev_gluing}) and the fact that the trace operator commutes with the conformal mapping  (Lemma~\ref{lem:trace_commutation}).

\begin{lemma} 
Suppose $u \in \mc E(\m C)$ and $\eta$ is a chord-arc curve in $\hat{\mathbb{C}}$. The Jonsson-Wallin trace of $u$ on $\eta$ is defined for arclength a.e. $z \in \eta$ by the following limit of averages
\begin{equation}\label{def:trace}
\mc R_{\eta}[u](z):=\lim_{r \to 0+}u_{B(z,r)}, 
\end{equation}
where $B(z,r) = \{w: |w-z| < r\}$. Moreover, $\mc R_{\eta}[u] \in H^{1/2}(\eta)$.
\end{lemma}

\begin{proof}
Assume first that $\eta$ is bounded.
Then without loss of generality (by localization and the Poincar\'e inequality), we may assume that $u \in W^{1,2} (\m C)$. 
Since $\eta$ is chord-arc, it follows from \cite{JW1984} Theorem VII.1, p.182,  that $\mc R_{\eta} [u] \in H^{1/2} (\eta)$.

This extends to the restriction of $\mc E(\m C)$ on $\eta$ passing through $\infty$ via a M\"obius transformation. Indeed, let $m$ be a M\"obius transformation such that $m (\eta)$ is bounded.
By conformal invariance of the Dirichlet energy, $u \circ m^{-1} \in \mc E (\m C)$. Therefore 
$\mc R_{m (\eta)}  [u \circ m^{-1}] \in H^{1/2} (m (\eta))$. 
For $z \in \eta \cap \m C$, 
$$ \mc R_{\eta} [u] (z)   = \mc R_{m (\eta)} [ u \circ m^{-1}] \circ m (z), $$
since $m$ is smooth in a neighborhood of $z$.
Hence, we have
$$\norm{ \mc R_\eta [u] } _{H^{1/2} (\eta)} = \norm { \mc R_{m (\eta)} [u \circ m^{-1}]\circ m }_{H^{1/2} (\eta)} = \norm { \mc R_{m (\eta)} [u \circ m^{-1}]}_{H^{1/2} (m(\eta))} < \infty, $$
where the second equality follows from \eqref{eq:norm_1/2_mobius_invariant}. 
\end{proof}

The Jonsson-Wallin trace is also defined without ambiguity from one side of the curve $\eta$:   
From Lemma~\ref{lem:jones-extension}, 
functions in $\mc E (\O)$ can be extended to $\mc E (\m C)$.
For  $ u \in \mc E (\m C)$, we will denote by $ u|_{\O}$ the restriction of $u$ to the domain $\O$.

\begin{lemma}[\cite{BM3} Theorem 5.1] \label{lem:trace_extension}
Let  $u \in \mc E (\O)$ and  $\tilde u \in \mc E (\m C)$ such that $\tilde u|_{\O} =u$.
The operator $\mc R_{\O \to \eta}$ defined as
\begin{equation*}
\mc R_{\O \to \eta}[u](z) := \mc R_{ \eta}[\tilde u] (z) ,\quad \text{ for arclength a.e. }z\in \eta,
\end{equation*}
  does not depend on the choice of the extension $\tilde u \in \mc E (\m C)$. 
\end{lemma}

The lemma is proved in \cite{BM3} 
for $u \in W^{1,2}$. The passage from $W^{1,2}$ to $\mc E(\m C)$ follows from a standard localization argument:  
    Let $v, w \in \mc E (\m C)$ such that 
    $v|_{\O} = w|_{\O} = u$. 
    For $z \in \eta$, let $\rho$ be a smooth function supported in $B(z, 2)$ which equals $1$ in $B(z, 1)$.
    We have $\rho v \in W^{1,2} (\O)$. Moreover $\rho v|_{\O} = \rho w|_{\O}$.
    Applying $\mc R_{\eta}$ to $\rho v$ and $\rho w$, there is no ambiguity in defining the trace for $\rho u$, and we get
    $$\mc R_{\eta} [\rho v] (y) = \mc R_{\eta} [\rho w] (y), \quad \text{for a.e. }  y \in B(z,1) \cap \eta.$$
    Since $\rho \equiv 1$ in a neighborhood of $z$, we have 
    $\mc R_{\eta} [v] (y) = \mc R_{\eta} [w] (y)$.

The lemma below states that $W^{1, 2}_{0}(\O)$ is exactly the kernel of $\mc R_{\O \to \eta}$ which also coincides with Sobolev functions that can be extended by $0$.  
Let $\O^*$ be the connected component of $\m C \smallsetminus \eta$ different from $\O$.

\begin{lemma}[\cite{BM3} Cor.~5.4~eq.~(5.37), Lemma~5.10] \label{lem:vanishing_trace}
For $u \in W^{1,2} (\O)$, 
if we denote by 
$\tilde u$ the function such that $\tilde u|_{\O} = u$ and $\tilde u|_{\O^*} = 0$, then
$$W^{1,2}_{0}(\O) = \{u \in W^{1,2}(\O): \mc R_{\O \to \eta} [u] =0 \, \text{a.e.}\} = \{ u \in W^{1,2} (\O): \tilde u\in W^{1,2} (\m C)\}.$$
\end{lemma}

\begin{lemma}[Gluing]\label{lem:sobolev_gluing}
If $u \in \mc E (\O)$ and $v \in \mc E (\O^*)$  have matching trace along $\eta = \partial \O = \partial \O^*$, that is, if
$$\mc R_{\O \to \eta} [u] (z)= \mc R_{\O^* \to \eta} [v](z) \quad \text{a.e. } z \in \eta, $$
then there exists a function $w \in \mc E(\m C)$ such that $w|_\O = u$ and $w|_{\O^*} = v$.
\end{lemma}
\begin{proof}
Using a partition of unity, we may assume $u \in W^{1,2} (\O)$ and $ v \in W^{1,2}(\O^*)$. 
  Lemma~\ref{lem:jones-extension} implies that there exists $\tilde v \in W^{1,2} (\m C)$ such that $\tilde v|_{\O^*} = v$ and for a.e. $z \in \eta$,
  $$\mc R_{\O^*  \to \eta} [v] (z) = \mc R_{\eta} [\tilde v](z) = \mc R_{\O \to \eta} [\tilde v] (z).$$
  Therefore $\mc R_{\O  \to \eta} [u - \tilde v]=0$ a.e. on $\eta$. Note that $(u- \tilde v)|_{\O} \in W^{1, 2}_0(\O) \xhookrightarrow{} W^{1, 2} (\m C) $ and we let $\phi \in W^{1, 2}(\m C)$ denote the extension of $(u - \tilde v)|_{\O}$ by zero.
 We set $w:=\phi + \tilde v \in W^{1, 2}(\m C)$ which extends both $u$ and $v$ in the desired way. 
\end{proof}

We will now relate the Jonsson-Wallin trace to the function obtained by taking non-tangential limits on $\eta$.
Let $\alpha > 0$ be given. We define the non-tangential approach region to $\zeta \in \eta = \partial \Omega$ (relative to $\Omega$) by \[A_\alpha(\zeta) = \{z \in \Omega : |z-\zeta| \le (1+\alpha) \dist(z, \eta)\}.\]
Since $\eta$ is chord-arc, $A_\alpha(\zeta)$ contains a path tending to $\zeta$ for all sufficiently large $\alpha$.
A function $f : \Omega \to \mathbb{C}$ is said to have \emph{non-tangential limit} $w$ at $\zeta$ if for all $\alpha$ large enough the limit of $f$ along any path in $A_\alpha(\zeta)$ tending to $\zeta$ equals $w$. Conformal maps between quasidisks preserve non-tangential approach regions (with quantitative bounds on constants), so taking a non-tangential limit commutes with applying the Riemann map in our setting. See, e.g., Proposition~1.1 of \cite{jerison-kenig82}.

Note that for  $z \in A_{\a}(\zeta)$, we have $B(z, C_\a |z -\zeta| ) \subset A_{2\a} (\zeta)$, where $C_\a = \a/(2 (1+\a)^2)$. In particular, if there exists a path in $A_{\a}(\zeta)$ tending to $\zeta$, for all $r >0$,
\begin{equation}\label{eq:cone_condition}
|A_{2\a} (\zeta) \cap B(\zeta, 2r)| \ge |B(z_r, C_\a r)| \ge c r^2,
\end{equation}
where $z_r$ is a point on the path in $A_\a (\zeta)$ with $|z_r - \zeta| = r$ and $c >0$ is independent of $r$.

The next lemma shows that the Jonsson-Wallin trace of a function in the harmonic Dirichlet space on a chord-arc domain coincides with its non-tangential limits.

Given $ u \in \mc E_{\text{harm}} (\Omega)$, by Lemma~\ref{lem:trace_extension} we may extend $u$ to a function in $\mc E(\mathbb{C})$ and the Jonsson-Wallin trace is independent of the particular extension chosen so we may write $\mc R_{\eta}[u]$ unambiguously.

\begin{lemma} \label{lem:traces_are_identical}
Suppose $\eta=\partial \Omega$ is a chord-arc curve in $\hat{\mathbb{C}}$. 
Let $ u \in \mc E_{\text{harm}} (\Omega)$. 
 Then for arclength almost every $\zeta \in \eta$, the non-tangential limit of $u$ at $\zeta$ exists and agrees with $\mc R_{\eta}[u](\zeta)$ as defined in \eqref{def:trace}.  
\end{lemma}

\begin{proof}
In the case $\Omega = \mathbb{D}$ we know that $u$ has non-tangential limits almost everywhere and the limiting function lies in $H^{1/2}(\mathbb{T})$. It then follows from the fact that harmonic measure and arclength are mutually absolutely continuous on $\eta$ that $u$ has non-tangential limits almost everywhere on $\eta$.

Let $\zeta \in \eta$ be a point such that both the Jonsson-Wallin trace and non-tangential limit exist at $\zeta$. Let $\ee > 0$ be given. For $r > 0$, let $u_r =  |B_r|^{-1} \int_{B_r} u(z) dz^2$ where $B_r:=B(\zeta,r)$. (Recall that we consider an extension of $u$.) By the Cauchy-Schwarz and Poincar\'e inqualities, 
\[\frac{1}{|B_r|}\int_{B_r} |u(z)-u_r| dz^2 \le c E_r, \quad \textrm{for} \quad E_r := \left(\int_{B_r}|\nabla u(z)|^2 dz^2\right)^{1/2}.\]
Since $u \in \mathcal{E}(\mathbb{C})$, we have $E_r = o(1)$ and so for $r$ small enough,
\begin{equation}\label{feb2.1}
|\{z \in B_r: |u(z)-u_r| > \epsilon \}| \le c |B_r| E_r/\ee < |B_r| \epsilon,
\end{equation}
where we used Markov's inequality for the first bound.
Hence, since $ \mc R_{\eta}[u](\zeta) = \lim_{r \to 0+} u_r$, taking $r$ smaller if necessary, we have 
\[
|\{z \in B_r: |u(z) -  \mc R_{\eta}[u](\zeta)| \le 2\ee\}| \ge (1-\ee) |B_r|.
\]
By \eqref{eq:cone_condition} there exist $\alpha, c_1>0$ such that
\begin{equation}\label{feb2.2}
|A_\alpha(\zeta) \cap B_r| \ge c_1 |B_r|,
\end{equation}
for all $r>0$ sufficiently small,
where $A_\alpha(\zeta) = \{z \in \Omega : |z-\zeta| \le (1+\alpha) \dist(z, \eta)\}$ is the non-tangential approach region at $\zeta$. Therefore, using \eqref{feb2.2} and if $\ee > 0$ is taken sufficiently small, \eqref{feb2.1} shows that the set $\{z \in B_r: |u(z) -  \mc R_{\eta}[u](\zeta)| \le 2\ee\}$ and $A_\alpha(\zeta)$ must intersect for all sufficiently small $r$ and so the limit taken in $A_\alpha(\zeta)$ equals $ \mc R_{\eta}[u](\zeta)$, as desired. 
\end{proof}

\begin{lemma}\label{lem:trace_commutation}
Suppose $\partial \Omega$ is a chord-arc curve in $\hat{\mathbb{C}}$.  Let $u \in \mc E(\Omega)$ and suppose $\phi: \m D \to \O$ is a Riemann mapping. Then,
\begin{equation} \label{eq:trace_commutation}
\mc R_{ \m T} [u \circ \phi] = \mc R_{\eta} [u] \circ \phi \in  H^{1/2} (\m T).
\end{equation}
\end{lemma}
\begin{proof}
   Assume first that   $u \in \mc E_{\textrm{harm}}(\Omega)$.
   Since taking non-tangential limits commute with applying $\phi$ and since harmonic measure and arclength are mutually absolutely continuous, \eqref{eq:trace_commutation} follows using Lemma~\ref{lem:traces_are_identical}.

  Let $u \in \mc E(\Omega)$.  We first assume that $\Omega$ is bounded. Then we have $u \in W^{1,2} (\O)$. Next, let $u_0 \in W^{1,2}_0 (\O)$ and $u_{h} \in \mc E_{\text{harm}} (\O)$ such that $u = u_0 + u_h$.
  It follows from conformal invariance of the Dirichlet energy and from the Poincar\'e inequality that the operator $W^{1,2} (\O ) \to W^{1,2} (\m D)$, $ u \mapsto  u \circ \phi$ and its inverse are bounded. 
  Since $W^{1,2}_0 (\m D)$ is the closure of $C^{\infty}_c(\m D)$, 
  we have $W^{1,2}_0 (\m D) = W^{1,2}_0 (\O) \circ \phi$. In particular, 
  $$v \in W^{1,2}_0 (\O) \, \Leftrightarrow \, v \circ \phi \in W^{1,2}_0 (\m D) \, \Leftrightarrow  \, \mc R_{ \m T} [v \circ \varphi] = 0 \text{ a.e. on } \m T$$
from Lemma~\ref{lem:vanishing_trace}.
  Therefore
$
  \mc R_{\eta} [u_0] \circ \phi = 0 = \mc R_{ \m T}[u_0  \circ \phi] \text{ a.e. on } \m T. 
 $
We conclude by applying \eqref{eq:trace_commutation} to $u_{h}$.

When $\O$ is unbounded, let $m$ be a M\"obius transformation such that $m (\O)$ is bounded. 
Since $m$ is smooth in a neighborhood of $\eta$, $\mc R_{m (\eta)} [u \circ m^{-1}] \circ m = \mc R_{\eta} [u]$, and we have a.e. on~$\m T$,
$$ \mc R_{\eta} [u] \circ \phi  =  \mc R_{m (\eta)} [u \circ m^{-1}] (m \circ \phi ) = \mc R_{\m T} [u \circ \varphi]. $$
The second equality follows from applying \eqref{eq:trace_commutation} to $m(\O)$ for $\tilde u = u \circ m^{-1}$, $\tilde \phi = m \circ \phi$.
  \end{proof}

\section{Density of $C^{\infty}_c (\m C)$ in $\mc E (\m C)$}
Here we provide a proof of the fact that test functions are dense in the homogeneous Sobolev space $\mc E(\m C)$ (based on a write-up by Alexis Michelat). The result must be well-known, but we were not able to locate a precise reference in the literature.

Let $(\rho_j)_{j \in \m N}$ be a family of mollifiers such that for all $f \in L^2(\m C)$, 
$$ \lim_{j \to \infty}\norm{\rho_j * f - f}_{L^2(\m C)} = 0 $$ 
and $(\eta_j)_{j \in \m N}$ a family of smooth function supported in $B(0, 2^{j+1})$, such that $\eta_j \equiv 1$ in $B(0, 2^j)$, $0 \le \eta\le 1$, and $\norm{\nabla \eta_j}_{L^{\infty}} \le 1/2^j$. Let $A_j$ denote the annulus $B(0, 2^{j+1}) \smallsetminus B(0, 2^j)$. Recall that we write $u_{A_j}$ for $\frac{1}{|A_j|}\int_{A_j} u dz^2$.
\begin{lemma}\label{lem:density}
   For $u \in \mc E (\m C)$, let $u_j = \rho_j * \left(\eta_j (u - u_{A_j}) \right) \in C^{\infty}_c (\m C)$. We have
   $$\lim_{j \to \infty}\norm{\nabla u - \nabla u_j}_{L^2(\m C)}=0.$$
\end{lemma}
\begin{proof} 
We have
$$\nabla u_j = \rho_j * (\nabla \eta_j (u - u_{A_j}) ) + \rho_j * (\eta_j \nabla u).$$
Using Young's convolution inequality we see that
$$\norm{\rho_j * (\nabla \eta_j (u - u_{A_j}))}_{L^2(\m C)} \le \norm{\rho_j}_{L^1(\m C)} \norm{\nabla \eta_j (u - u_{A_j})}_{L^2(\m C)} = \norm{\nabla \eta_j (u - u_{A_j})}_{L^2(\m C)}, $$
since $\rho_j$ is a mollifier. By the Poincar\'e inequality in $A_j$ and the fact that $\norm{\nabla \eta_j}_{L^{\infty}} \le 1/2^j$ is supported in $A_j$, there is $C < \infty$ independent of $j$, such that
$$\norm{\nabla \eta_j (u - u_{A_j}) }_{L^2(\m C)} \le \frac{1}{2^j} \norm{u - u_{A_j} }_{L^2(A_j)} \le  \frac{C \, \diam(A_j)}{2^j}\norm{\nabla u}_{L^2 (A_j)} $$
and the right-hand side tends to $0$ as $j \to \infty$.
On the other hand,
$$\lim_{j \to 0}\norm{\nabla u - \rho_j * (\eta_j \nabla u)}_{L^2(\m C)} = 0 $$
and this concludes the proof.
\end{proof}

\end{document}

%% file: yilin.tex
\usepackage{url}
\newcommand{\abs}[1]{\left\lvert #1 \right \rvert}

\newcommand{\brac}[1]{\left \langle #1 \right \rangle}
\newcommand{\norm}[1]{\| #1 \|}

\newcommand{\mc}[1]{\mathcal{#1}}
\newcommand{\m}[1]{\mathbb{#1}}

\renewcommand\Re{\operatorname{Re}}
\renewcommand\Im{\operatorname{Im}}

\def\SLE{\operatorname{SLE}}

\def\a{\alpha}

\def\g{\gamma}

\def\d{\delta}
\def\D{\Delta}

\def\k{\kappa}

\def\s{\sigma}

\def\O{\Omega}
\def\vare{\varepsilon}
\def\HH{{\mathbb H}}
\def\Chat{\hat{\m{C}}}

\def\diam{{\rm diam}}
\def\dist{{\rm dist}}
\def\SLE{\operatorname{SLE}}
\def\GFF{\operatorname{GFF}}

\newcommand{\ad}[1]{\overline{#1}}